\documentclass[11pt,reqno]{amsart}

\usepackage{amsmath,amssymb,color,accents}
\usepackage{graphicx,color}
\usepackage[colorlinks=true, pdfstartview=FitV, linkcolor=magenta, citecolor=magenta, urlcolor=blue]{hyperref}
\usepackage[nameinlink,capitalise]{cleveref}
\usepackage{comment}
\usepackage{mathtools}
\usepackage[all]{xy}
\usepackage{ytableau}
\usepackage{amsfonts}
\usepackage{longtable}
\usepackage{enumerate}
\usepackage{tikz}
\usepackage{multicol}
\usepackage{tabularx}
\usepackage{enumitem}

\usetikzlibrary{decorations.pathmorphing,shapes}
\tikzset{snake it/.style={decorate, decoration=snake}}

\usepackage{tikz}
\usepackage[all]{xy}
\usetikzlibrary{matrix,cd}

\tikzset{
  treenode/.style = {align=center, inner sep=0pt, text centered,solid,thin,
    font=\sffamily},
  arn_n/.style = {treenode, circle, white, font=\sffamily\bfseries, draw=black,
    fill=black, text width=.5em},
  arn_nl/.style = {treenode, circle, white, font=\sffamily\bfseries, draw=black,
    fill=black, text width=1.5em},  
  arn_r/.style = {treenode, circle, red, draw=red, 
    text width=.5em, very thick},
  arn_v/.style = {treenode, circle, black, font=\sffamily\bfseries, draw=black, text width=1.2em},
  arn_x/.style = {treenode, rectangle, draw=black,
    minimum width=.5em, minimum height=0.5em},
  dott/.style={edge from parent/.style={dotted, very thick,circle,draw}},
  emph/.style={edge from parent/.style={dashed, very thick,circle,draw}},
  norm/.style={edge from parent/.style={solid,thin,circle,draw}}
}

\usepackage{tabstackengine}

\usepackage{listings}

\let\amsamp=&
\usepackage{bbm}
\newcommand{\bigzero}{\mbox{\normalfont\Large\bfseries 0}}

\newcommand{\sbm}[1]{{\let\amp=&\begin{pmatrix}#1\end{pmatrix}}}

\numberwithin{equation}{section}
\newcommand{\wps}{weighted projective space }
\newcommand{\gdp}{general double point }

\newtheorem{theorem}{Theorem}[section]
\newtheorem{lemma}[theorem]{Lemma}
\newtheorem{proposition}[theorem]{Proposition}
\newtheorem{corollary}[theorem]{Corollary}

\newtheorem{problem}[theorem]{Problem}
\newtheorem*{theorem*}{Theorem}
\newtheorem*{theorem123}{Theorem \ref{AH123}}
\newtheorem*{theoremExcep}{Theorem \ref{Theorem: 1bc bound}}
\newtheorem*{theoremKpoints}{Proposition \ref{suffkpt}}

\newtheorem*{theoremTERA}{Theorem \ref{TERA}}
\newtheorem*{LemmaChandler}{Lemma \ref{C13}}
\newtheorem*{TheoremSecantVariety}{Theorem \ref{equiv}}
\newtheorem*{WeightedLine}{Proposition \ref{GK}}
\newtheorem*{SimplePointInterpolation}{Theorem \ref{prop: simple point interpolation}}
\newtheorem*{NECK}{Proposition \ref{NECK}}

\theoremstyle{definition}
\newtheorem{definition}[theorem]{Definition}

\theoremstyle{remark}
\newtheorem{remark}[theorem]{Remark}
\newtheorem{example}[theorem]{Example}

\newtheorem{notation}[theorem]{Notation}

\newcommand{\bpf}{\begin{proof}}
\newcommand{\epf}{\end{proof}}
\newcommand{\bpr}{\begin{proposition}}
\newcommand{\epr}{\end{proposition}}
\newcommand{\bdf}{\begin{definition}}
\newcommand{\edf}{\end{definition}}
\newcommand{\blm}{\begin{lemma}}
\newcommand{\elm}{\end{lemma}}
\newcommand{\bex}{\begin{example}\rm }
\newcommand{\eex}{\end{example}}
\newcommand{\bcor}{\begin{corollary}}
\newcommand{\ecor}{\end{corollary}}
\newcommand{\bthm}{\begin{theorem}}
\newcommand{\ethm}{\end{theorem}}
\newcommand{\be}{\begin{enumerate}}
\newcommand{\ee}{\end{enumerate}}
\newcommand{\bq}{\begin{equation}}
\newcommand{\eq}{\end{equation}}
\newcommand{\bb}{\begin{itemize}}
\newcommand{\eb}{\end{itemize}}
\newcommand{\bpw}{\begin{cases}}
\newcommand{\epw}{\end{cases}}

\DeclareMathOperator{\ev}{ev}
\DeclareMathOperator{\Proj}{Proj}

\DeclareMathOperator{\lcm}{lcm}

\DeclareMathOperator{\Spec}{Spec}

\DeclareMathOperator{\Hom}{Hom}
\DeclareMathOperator{\gr}{gr}
\DeclareMathOperator{\Char}{char}
\DeclareMathOperator{\Min}{Min}

\newcommand{\fp}{{\mathfrak p}}
\newcommand{\bs}{\backslash}
\renewcommand{\d}{{\rm d}}

\newcommand{\fm}{{\mathfrak m}}

\newcommand{\N}{{\mathbb{N}}}

\newcommand{\T}{{\mathbb{T}}}
\newcommand{\Z}{{\mathbb Z}}
\renewcommand{\AA}{{\mathbb A}}
\newcommand{\PP}{{\mathbb P}}
\numberwithin{equation}{section} 
\newcommand{\AH}{\text{AH}}
\newcommand{\sat}{\text{sat}}
\newcommand{\symb}{^{(2)}}
\newcommand{\LF}{\left\lfloor}
\newcommand{\RF}{\right\rfloor}
\newcommand{\RC}{\right\rceil}
\newcommand{\LC}{\left\lceil}

\title{
Interpolation in Weighted Projective Spaces}

\author[S. ROSHAN-ZAMIR]{Shahriyar Roshan-Zamir}
\address{Department of Mathematics, University of Nebraska-Lincoln, 1144 T St
Lincoln, NE 68588}
\email{sroshanzamir2@huskers.unl.edu}

\begin{document}

\thanks{The author received support from the National Science Foundation  grant DMS-210225.}

\thanks{Mathematics Subject Classification: 13A02, 13F20, 14N05.}
\thanks{Key Words: Interpolation, Hilbert Function, Weighted Projective Space, Non-standard Grading.}

\begin{abstract}
Over an algebraically closed field, the {\it double point interpolation} problem asks for the vector space dimension of the projective hypersurfaces of degree $d$ singular at a given set of points. 
After being open for 90 years, a series of papers by J. Alexander and A. Hirschowitz in 1992--1995 settled this question in what is referred to as the Alexander-Hirschowitz theorem. In this paper we primarily use commutative algebra to lay the groundwork necessary to prove analogous statements in the {\it weighted projective space}, a natural generalization of the projective space. We prove the Hilbert function of general simple points in any $n$-dimensional weighted projective space exhibits the expected behavior. We also introduce an inductive procedure for weighted projective space, similar to that originally due to A. Terracini from 1915, to demonstrate an example of a weighted projective plane where the analogue of the Alexander-Hirschowitz theorem holds without exceptions and prove our example is the only such plane. Furthermore, Terracini's lemma regarding secant varieties is adapted to give an interpolation bound for an infinite family of weighted projective planes. 
\end{abstract}

\maketitle

\setcounter{tocdepth}{1} 
\tableofcontents

\section{Introduction}

Let $k$ be an algebraically closed field. Consider  the vector space of homogeneous polynomials of degree $d$ in $n+1$ variables having singularities at $r$ general points $p_1,\ldots, p_r$ of the projective space $\PP^n_k$.  For a polynomial $f$ to be a member of this vector space it must satisfy $r(n+1)$ linear conditions on its coefficients imposed by the vanishing of all partial derivatives of $f$ at $p_1,\ldots, p_r$. It is natural to inquire whether these conditions are independent. This problem was considered by A.~Hirschowitz and J.~E.~Alexander in a series of papers \cite{Hirschowitz85, Alexander88, AH92b, AH92a, AH95} and finally settled in the form stated below, where $H_d(-)$ denotes the dimension of the vector subspace of homogeneous forms of degree $d$ of the input.

\bthm[Alexander-Hirschowitz]
\label{thm: AH}
Let $n,d$ be positive integers. Let $X$ be a general set of $r$ double points in $\PP^n_k$ and $I={I_{X}}$ be its defining ideal. Then 
\begin{equation} \label{eq: AH main}
    H_d(R/I)=\min \left \{r(n+1), {{n+d}\choose{d}} \right \}
\end{equation}
with exceptions:
\begin{itemize}
    \item $d=2$ and $2 \leq r \leq n$;
    \item $d=3$, $n=4$ and $r=7$; and
    \item $d=4$, $2 \leq n \leq 4$ and $r={n+2 \choose 2}-1$.
\end{itemize}
\ethm 
We say that a set of double points $X$ has {\em expected dimension in degree $d$} if its ideal $I$ satisfies \Cref{eq: AH main}. \Cref{thm: AH} has two equivalent formulations in terms of the dimension of higher secant varieties to the Veronese embedding of $\PP^n$ and the Big Waring Rank problem (\cite[Appendix A]{HM21}) as well as applications in  numerical algebraic geometry \cite{NumericalAG}, deep polynomial neural networks \cite{NeuralNetwork}, and algebraic statistics \cite[Section 6.2]{AlgebraicStatistics}. The original proof of \Cref{thm: AH} is based on a degeneration technique called the ``la m\'{e}thode d'Horace diff\'{e}rentialle'', or m\'{e}thode d'Horace for short, which is based on an earlier degeneration technique by A. Terracini (\Cref{TERA}).
M\'{e}thode d'Horace was later simplified by K.~A.~Chandler \cite{Chandler01, Chandler02}. Further expositions and surveys on the Alexander-Hirschowitz theorem and related problems can be found in \cite{BO08, HM21}. 

Our main goal is to give an analogous statement to \Cref{thm: AH} in the weighted projective space denoted $\PP(a_0,\ldots,a_n)$, where each $a_i$ is a positive integer. This space, surveyed in \cite{DolgachevWPS, WPSBeltRobb}, arises naturally in contexts such as toric geometry \cite{WPSToric}, algebraic topology \cite{WPSAlgTop}, and $K$-theory \cite{WPSKTheory}.  When $a_i=1$ for all $i$, one can use \Cref{def: wps} to recover the  usual or straight projective space. The author's main approach has been to modify existing techniques from the usual projective space to fit this new setting. Such adaptions require careful, and often non-trivial, considerations of details. This is best exemplified in \Cref{App: Secant Variety}.

In \Cref{section: wps} we introduce the weighted projective space, its basic properties and the defining ideal of points. \Cref{ptsplne} shows that unlike the projective space, even a general point in the weighted projective plane need not be generated by a complete intersection ideal; a fact which is critical to computing $I$ in \Cref{thm: AH}. However, having a weight of $1$ forces the ideals of general points to be complete intersections, \Cref{POINTS ARE CI}, hence allowing for easy calculations of their ideals. 

\Cref{s:line} concerns an analogue to \Cref{thm: AH} in the weighted projective line and shows considering this property passes a litmus test, it holds in dimension $1$. Its main result is \Cref{GK} which we now state. 
\begin{WeightedLine}
A general set of points with mixed multiplicity in $\PP(a,b)$, where $\gcd(a,b)=1$, has expected dimension in degree $d$ for all $d\in \N$.    
\end{WeightedLine}
Here, mixed multiplicity refers to having different orders of singularity at the points. While for $n=1$ \Cref{thm: AH} does not require the assumption that the points are general, in contrast, our result excludes the point $[0:1]$, hence the hypothesis of \cref{GK} requires general points.

\Cref{section: interpolation in \wps} is dedicated to proving \Cref{prop: simple point interpolation} which concerns simple point interpolation in $\PP(a_0,a_1,\ldots,a_n)$. 
The precise result is as follows.
\begin{SimplePointInterpolation}
Let $X \subseteq \PP(a_0,a_1,\ldots,a_n)$ be a set of $r$ general simple points with defining ideal $I_X$. Then for all $d \in \N$ we have
\begin{equation*} 
    H_d(S/I_X)=\min \{ r, H_d(S)\}.
\end{equation*}
\end{SimplePointInterpolation}
The proof of \Cref{prop: simple point interpolation} in the usual projective space is accessible to an advanced undergraduate student; see \cite{JackHuizenga}. But a proof in the setting above is non-trivial and requires tempering with parameter spaces. 

In \Cref{section: weighted projective plane} we study sets of double points in weighted projective planes of the form $\PP(1,b,c)$. 
The main novelty of this section is taking a different view point towards the problem: instead of studying interpolation in a {\it fixed space}, we start with a \textit{fixed size} for a set of general double points, say $r$, and find examples of spaces in which $r$ general double points have expected dimension in every degree $d$. \Cref{NECK} proves a necessary condition for this to hold.
\begin{NECK}
If $r$ general double points in $\PP(1,b,c)$ have expected dimension in every degree, then $c<(r+1)b$ for $b\neq 1$ and $c\leq r+1$ when $b=1$.
\end{NECK}

The main result of \cref{section: weighted projective plane} is a partial converse to \Cref{NECK}.
\begin{theoremKpoints}    
Let $r \in \N$. If $rb<c<(r+1)b$ then a set of $r$ general double points in $\PP(1,b,c)$ has expected dimension in every degree $d$.
\end{theoremKpoints}

The entirety of \Cref{123} is dedicated to proving an Alexander--Hirschowitz theorem in $\PP(1,2,3)$. 
\begin{theorem123}
    Let $X \subseteq \PP(1,2,3)$ be a set of $r$ general double points. Then $X$ has expected dimension in every degree $d$ with no exceptions. Moreover $\PP(1,2,3)$ is the only weighted projective space where no exceptions occur.
\end{theorem123}

The most striking part of \Cref{AH123} is having {\it no exceptions} which is not the case even in the usual $\PP^2$. Indeed, \Cref{thm: 123 ONLY} shows $\PP(1,2,3)$ is the only such plane. The main inductive procedure of the proof, \Cref{TERA}, is the predecessor to m\'{e}thode d'Horace; a specialization technique due to Terracini from 1915, \cite{Terra3}. In \Cref{TERA} we observe Terracini's method can be generalized to $\PP(1,\ldots,a_n)$. The difference in numerology is key; $\PP_k^n$ contains only one hyperplane up to isomorphism. But in $\PP(1,\ldots,a_n)$ there are $n+1$ hyperplanes with different numerical characteristics.

\Cref{thm: AH} provides a finite list of exceptions in $\PP^2$; $2$ points in degree $2$ and $5$ points in degree $4$. In \Cref{example: deficiency} we show that general double points, even in a ``nice'' space such as $\PP(1,b,c)$ can have erratic behavior. Nonetheless, \Cref{Theorem: 1bc bound} gives a linear bound, in terms of the weights, on the degree $d$, after which a set of points of any size expected dimension.
\begin{theoremExcep}
    Let $X \subseteq \PP(1,b,c)$ be set of $r$ general double points. Then $X$ has expected dimension in degree $d$ if $d\geq 10c$ or if $d\geq 6c$ and $\LF \frac{2c}{b} \RF \geq 5$.
\end{theoremExcep}
The case of the projective plane  of \Cref{thm: AH} had been previously elucidated by Terracini in \cite{Terra2} and by Platini in \cite{Platini1}. 
Our proof of \Cref{Theorem: 1bc bound} mimics Terracini's approach. 
However the bound in our result depends on the weights. This is a case where the numerics of \wps work to our disadvantage: as the weights increase, one has more cases to check.

The essential ingredient in proving \Cref{Theorem: 1bc bound} is the equivalence of double point interpolation and computing the dimension of the higher secant variety to the Veronese embedding of projective space. 
This is made precise in \Cref{equiv}, the proof of which is relegated to \Cref{App: Secant Variety}. The main requisite in proving this equivalence is due to Lasker  \Cref{LASKERR}. 

We end the introduction highlighting three major themes of this paper: the divergence of notions from the usual projective space when they are considered in the weighted projective space, the difference in numerology and the importance of a weight of $1$. For example, in \Cref{def: Veronese Map}, one observes when $a_i\neq 1$ for all $i$, the $d$-th Veronese map is not defined on all of $\PP(a_0,\ldots,a_n)$, hence it is not an embedding; a fact crucial to the analysis of its secant variety. 
Other examples of this divergence are the defining ideal of a simple point not being complete intersection, addressed earlier, or stabilization of the Hilbert function at the expected number of independent conditions; see \Cref{remark: Weight of 1}~(2). In all cases, introducing a single weight of $1$ controls this divergence to some extent. All specialization techniques used for the interpolation problem rely heavily on numerical conditions \cite{AH95, Chandler01, BO08, AHElisaPostinghel}. While for a standard graded polynomial ring the Hilbert function can be computed using a binomial coefficient, no closed formula exists in the non-standard graded setting beyond $2$ variables. Therefore applying degeneration techniques becomes very tricky. The remarkable fact that the Hilbert function for the coordinate ring of $\PP(1,2,3)$ happens to have a closed formula, given in \Cref{FORMULA123},  allowed us to follow Terracini's method in this setting.


{\bf Acknowledgments:} I would like to thank my advisor, Dr.~Alexandra Seceleanu, for her support and infinite patience in providing detailed commentary on earlier versions of this paper. Her help was especially instrumental in the final version of \Cref{App: Secant Variety} and the proof of \Cref{lem: Chandler}. I thank Benjamin Briggs for spending numerous hours helping me with the gory details of abstract tangent spaces, his help was crucial in my initial proof of \Cref{equiv}. I thank Ola Sobieska for introducing me to \cite{CCD} and Jack Jeffries for his help fine tuning my lattice point counting skills. I thank Mahrud Sayrafi for his help with some of my Macaulay2 code. I have benefited enormously from my conversations with Jack Huizenga, Nawaj KC, Claudiu Raicu, Andrew Soto Levins, and Ryan Watson. Finally, I thank Paolo Mantero and Tài Huy Hà for writing \cite{HM21} and Giorgio Ottaviani and Maria Brambilla for writing \cite{BO08}, their papers are the main inspirations for this work.
\section{Weighted projective space and the interpolation problem} \label{section: wps}

\subsection{Weighted projective varieties}\label{s: wps}

Fix integers $a_0,\ldots, a_n\in \N$. 
\begin{definition} \label{def: wps}
The {\em weighted projective space} $\PP=\PP(a_0,\ldots, a_n)$ is a quotient of the affine space $\mathbb{A}_k^{n+1}$ under the equivalence relation
\begin{equation}\label{eq: equivalence relation} 
(p_0,\ldots ,p_n) \sim (\lambda^{a_0 }p_0,\ldots, \lambda^{a_n} p_n) \text{ for } \lambda \in k\setminus \{0\}.
\end{equation}
Denoting $k^\times=k\setminus\{0\}$ the multiplicative group of $k$, and viewing the above  as an action of $k^\times$ on $\mathbb{A}_k^{n+1}$ allows to write concisely
\begin{equation*}
\PP(a_0,\ldots ,a_n)=(\mathbb{A}_k^{n+1} \setminus \{0\})/k^\times.
\end{equation*}
\end{definition}
The usual (or straight) projective space $\PP_k^n$ can be recovered from this construction as $\PP(1,1,\ldots, 1)$. We reserve the notation $\PP_k^n$ for straight projective space and $\PP$ or $\PP(a_0,\ldots, a_n)$ for weighted projective space.

Let $S=k[x_0,\ldots, x_n]$ be a polynomial ring with (nonstandard) grading given by $\deg(x_i)=a_i$. Observe that endowing  $S$ with this grading is equivalent to giving an action on $\mathbb{A}_k^{n+1}$  as in \Cref{eq: equivalence relation}. Specifically, a polynomial $f\in S$ is  homogeneous of degree $d$ if and only if
\begin{equation}
\label{eq: action on polynomials}
f(\lambda^{a_0 }x_0,\ldots, \lambda^{a_n} x_n) = \lambda^d f(x_0,\ldots ,x_n) \text{ for all } \lambda\in k^\times.
\end{equation}
Therefore the condition $f(p) = 0$ is well defined on equivalence classes of $\sim$ in \Cref{eq: equivalence relation}.  One can thus alternatively define weighted projective space as the projective variety $\PP=\Proj(S)$. More generally, if $I$ is a homogeneous ideal of $S$, then the projective variety $\Proj (S/I)$ is the quotient  ${(V(I)\setminus\{0\})/k^\times}$. Throughout the paper $S$ always represents the (non-standard) graded homogenous polynomial ring of $\PP(a_0,\ldots,a_n)$.

Much like the usual projective space, weighted projective space is covered by the affine charts $U_i=\{p\in \PP\mid p_i\neq 0\}$. The coordinate ring of each $U_i$ is the ring of invariant polynomials $k[x_0, \ldots, x_{i-1},x_{i+1},\ldots, x_n]^{\mu_{a_i}}$, where the group $\mu_{a_i}$ is identified with the subgroup of $k^\times$ of $a_i$-th roots of 1 and the group acts  as follows:  an element  $\xi\in k$ with $\xi^{a_i}=1$ maps $x_j\mapsto \xi^{-a_j}x_j$. By \Cref{eq: action on polynomials} a polynomial is invariant under this action if and only if its degree is divisible by $a_i$ leading to a ring isomomorphism 
\begin{equation}\label{eq: affine charts}
k[x_0, \ldots, x_{i-1},x_{i+1},\ldots, x_n]^{\mu_{a_i}} \cong S[x_i^{-1}]_0, f\mapsto \frac{f}{x_i^{\deg(f)/a_i}}.
\end{equation}
Thus $U_i$ is $\Spec(k[x_0, \ldots, x_{i-1},x_{i+1},\ldots, x_n])/\mu_{a_i}=\mathbb{A}^n/\mu_{a_i}$.

\begin{definition}
The $d$-th {\em Veronese subring} of $S$ is the graded ring 
\begin{equation}\label{def: Veronese Subring}
S^{[d]}=\bigoplus_{i\geq 0} S_{id}.
\end{equation}
\end{definition}
Although $S^{[d]}$ is a proper subring of $S$, $\Proj(S^{[d]})$ and $\Proj(S)$ are isomorphic projective varieties, because any homogeneous ratio in $S$ can be expressed as a homogeneous ratio in $S^{[d]}$. There are two different conventions in use about degrees in $S^{[d]}$: one can let the elements of the Veronese subring have the same degree they had in $S$, or one can divide degrees through by $d$. The advantage of the latter approach is that there exists a Veronese subring  $S^{[D]}$ which is generated by its elements of the smallest degree (\cite[Lemma 2.1.6]{EGA2}); in this case, after dividing degrees by $D$, one can view  $S^{[D]}$ as being standard graded. This gives rise to an embedding $\PP\subset \PP^N_k$, where $N=\dim_k (S_D)-1$; see \Cref{rem: veronese embedding}.

Another application of this ideal is a reduction to sets of weights  that satisfy certain a simple condition which we now describe.
\begin{definition}\label{defn: well formed}
A weighted projective space $\PP(a_0, \ldots , a_n)$ is {\em well formed} if no $n-1$ of $a_0,a_1,\ldots,a_n$ have a common factor greater than $1$.
\end{definition}
\begin{remark} \label{rmk: well-formed} 
Any  \wps is isomorphic to a well formed weighted projective space. Two weighted projective spaces with the same weights are isomorphic up to a reordering of the weights. Henceforth the space $\PP(a_0,\ldots, a_n)$ is assumed to be well-formed with weights in {\it increasing} order $a_0\leq \cdots\leq a_n$ and $a_i \neq 1$ for some $0\leq i \leq n$.
\end{remark}

\subsection{Points in weighted projective space}

In this section we describe the ideal defining one point in the \wps and its properties.

\begin{definition}
Let $p\in\PP$ be a point. The defining ideal of $p$ is the ideal
\[
I_p=(f\in S\mid f \text{ homogeneous}, f(p)=0).
\]
\end{definition}

\begin{lemma}\label{lem: monomial curve}
Fix a point $p\in\PP$. Let $T=k[t]$ and consider the graded $k$-algebra homomorphism 
\[
\varphi_p:S\to T, x_i\mapsto p_it^{a_i}.
\]
Then $I_p=\ker (\varphi_p)$ and $S/I_p\cong k[t^{a_i} \mid p_i\neq 0]$. Moreover, if $p,q \in\PP$ are points such that for all $0\leq i\leq n$, $p_i=0$ if and only if $q_i=0$, then $I_p$ and $I_q$ are isomorphic $S$-modules.
\end{lemma}
\begin{proof}
Since $\varphi_p$ is a degree-preserving homomorphism, $\ker (\varphi_p)$ is a homogeneous ideal. If $f\in S$ is homogeneous of degree $d$ then $f\in \ker (\varphi_p)$ if and only if $f(p_it^{a_i})=0$ if and only if $t^df(p)=0$ if and only if $ f\in I_p$.

The map $\Theta:S\to S, \Theta(x_i)=\frac{p_i}{q_i}x_i$ (where we make the convention $\frac{0}{0} =1$) is an automorphism on $S$ if and only if  for all $0\leq i\leq n$, $p_i=0\iff q_i=0$. It satisfies $\Theta(I_p)=I_q$ thus demonstrating an isomorphism $I_p\cong I_q$. 
\end{proof}

In the case $p={\bf 1}=[1:1:\cdots:1]$ the ideal $I({\bf 1})=\ker (\varphi_{\bf 1})$ is called a {\em monomial curve ideal}. This class of ideals has been studied in \cite{Herzog}, \cite{Bre1} and \cite{KunzMonomial}. An expository survey of affine monomial curves and their computational aspects can be found in \cite{affinemonomialcurves}. The particular case of affine space curves is best understood. It was studied in \cite{Herzog}.

\bpr[{\cite{Herzog, Denham}}]\label{thm: Herzog}
Consider  positive integers $(a,b,c)$ with $\gcd(a,b,c)=1$. Let $r_i$ for $i=1,2,3$ be the smallest positive integer such that the following equations admit a solution in non-negative integers.
\begin{equation}\label{eq: Herzog}
    \begin{split}
        r_1a&=k_1b+g_1c\\
       r_2b&=k_2a+g_2c\\
      r_3c&=k_3a+g_3b
  \end{split}
\end{equation}
Then the kernel of the $k$-algebra map 
\[
\varphi:k[z,u,v]\to k[t], z\mapsto t^a, u\mapsto t^b, v\mapsto t^c
\]
 is the ideal
\[
I= (z^{r_1}-u^{k_1}v^{g_1}, u^{r_2}-z^{k_2}v^{g_2}, v^{r_2}-z^{k_3}u^{g_3}).
\]
Moreover two of the binomials above agree up to sign if and only if $0\in\{k_1, k_2, k_3, g_1, g_2, g_3\}$.
\epr

If in \Cref{eq: Herzog} $k_i=0$ or $g_i=0$ for any $i$, then we say the weights $(a,b,c)$ satisfy the \textit{Herzog Criteria} (HC).

\bpr \label{ptsplne} Consider a well formed weighted projective plane $\mathbb{P}(a,b,c)$ with coordinate ring $S=k[z,u,v]$. Then the following hold:
\begin{enumerate}
\item If $p_i\neq 0$ for $1\leq i\leq 3$ then
\begin{eqnarray*}
p=[0:p_1:p_2] & \Rightarrow & I_p=(z,p_2^bu^c-p_1^cv^b)\\
p'=[p_0:0:p_2] & \Rightarrow & I_{p'}=(u,p_2^az^c-p_0^cv^a)\\
p''=[p_0:p_1:0] & \Rightarrow & I_{p''}=(v,p_1^az^b-p_0^bu^a).
\end{eqnarray*}
\item Let $p=[p_0:p_1:p_2]$ with $p_i \neq 0$ for all $i$. Then 
\[
I_p=(p_1^{k_1}p_2^{g_1}z^{r_1}-p_0^{r_1}u^{k_1}v^{g_1}, p_0^{k_2}p_2^{g_2}u^{r_2}-p_1^{r_2}z^{k_2}v^{g_2},p_0^{k_3}p_1^{g_3}v^{r_3}-p_2^{r_3}z^{k_3}u^{g_3})
\]
where $r_i$, $k_i$ and $g_i$ are described in \Cref{eq: Herzog}
\end{enumerate}
\epr 

\bpf 
Let $p=[0:p_1:p_2]$. The polynomial $p_2^bu^c-p_1^cv^b$ is irreducible because the only homogeneous polynomials in $k[u,v]$ of degree less than $bc$ are monomials. Observe that $J=(z,p_2^bu^c-p_1^cv^b) \subseteq I_p$, and $J$ is a prime ideal of height 2. 
Since $I_p$ is also a prime ideal of height two, it follows that $J=I_p$. Other cases of $(1)$ are proven similarly.

Statement (2) follows from \Cref{lem: monomial curve} and \Cref{thm: Herzog}. Specifically, the proof of \Cref{lem: monomial curve} yields that $I_p=\Theta(I_{\bf 1})$ where $\Theta(x_i)=x_i/p_i$ and $I_{\bf 1}$ is described as $I$ in \Cref{thm: Herzog}.
\epf  

\begin{corollary} \label{POINTS ARE CI}
All ideals defining points in the weighted projective plane $\PP(a,b,c)$ are complete intersections if and only if $a,b,c$ satisfy (HC).
\end{corollary}

The next proposition shows that ideals defining points in certain higher dimensional weighted projective spaces are also complete intersections.

\bpr \label{defidplane} Consider $\PP=\mathbb{P}(1,1
,\ldots,1,a_{i+1},\ldots, a_n)$ with $a_j \neq 1$ for all $i+1\leq j \leq n$ and let  $S=k[x_0,\ldots,x_n]$ be the coordinate ring of $\PP$.  Let $p=[p_0: \ldots:p_n]$ where $p_t \neq 0$ for some $0\leq t \leq i$. Then 
\[
I_p=(p_0x_t^{a_0}-p_t^{a_0}x_0,\ldots,p_nx_t^{a_n}-p_t^{a_n}x_n).
\] 
\epr 
\bpf Note that $J=(p_1x_t^{a_1}-p_t^{a_1}x_1,\ldots,p_nx_t^{a_n}-p_t^{a_n}x_n) \subseteq I_p$ and $S/J \cong k[x_t]$ implies $J$ is a prime ideal. Further $\text{height}(J)=n$ because $J$ is a complete intersection with $n$ generators. Since $I_p$ is also a prime ideal of height $n$ the equality $J=I_p$ follows.
\epf
We end this section by noting that finding explicit generators for the defining ideal of a point in $\PP(a_0,\ldots,a_n)$ for $n\geq 3$ is much harder. In fact, outside of the case of \Cref{defidplane} we do not know what the minimal generators of ideals of points are.

\begin{problem}
    Find explicit generators for the defining ideal of a general point in
    $\PP(a_0,\ldots,a_n)$ for $n\geq 3$.
\end{problem}


\subsection{Hilbert function} 

Throughout $\N$ denotes the natural numbers (not including zero), $\N_0=\N\cup\{0\}$, $k$ is a field  and $S=k[x_0,\ldots, x_n]$ is a polynomial ring in variables $x_0,\ldots, x_n$ with coefficients in $k$.
We endow $S$ with a grading given by $\deg(x_i)=a_i\in\N$.

A graded $S$-module is a module that admits a decomposition $M=\bigoplus_{i\in\Z}M_i$, where $M_i$ denotes the $k$ vector space of homogeneous elements of $M$ of degree $i$. The {\em Hilbert function} of $M$ is the function $H_{-}(M):\N_0\to \N_0$ given by $H_i(M)=\dim_k M_i$.   

\begin{notation}\label{not: s_d}
Throughout the paper we denote the Hilbert function of the polynomial ring $S$ by $s_i=\dim_k S_i$.
\end{notation}
The integer $s_d$ is the number of monomials of degree $d$ in the polynomial ring $S$. Computing this number is equivalent to finding the number of non-negative integer solutions to the Diophantine equation
\begin{equation*}
    \sum_{i=0}^n a_it_i=d, \text{ where }  t_i\in \N_0 \text{ for all } 0\leq i\leq n.
\end{equation*}

The case of two variables has explicit formulas which are used repeatedly.

\begin{example} \label{prop: number of monomials in two variabels}
    Let $S=k[z,u]$ with $\deg(z)=a$ and $\deg(y)=b$ where $\gcd(a,b)=1$. Let $p$ and $q$ be positive integers such that $aq-bp=1$. By \cite[Corollary 1.6]{CCD} we have
    \begin{equation}\label{eq: sd for ab}
        s_d = \begin{cases}
        \LF \frac{qn}{b} \RF - \LF \frac{pn}{a} \RF   &\text{ if } a \nmid d\\
         \LF \frac{qn}{b} \RF - \frac{pn}{a} +1 &\text{ if } a \mid d.
    \end{cases}
    \end{equation}
In particular for $a=1$  this yields
    \begin{equation}\label{eq: sd for 1b}
         s_d=\LF \frac{d}{b} \RF +1.
    \end{equation} 
\end{example}

\begin{remark}\label{rmk: SES-NZD}
    For a  graded $S$-module $M$, where $x_i$ is a non-zero divisor of degree $a_i$ on $M$ we have the equation
    \begin{equation} \label{equation: HF-NZD}
        H_d(M)=H_{d-a_i}(M)+H_d(M/(x_i\cdot M)).
    \end{equation}
\Cref{equation: HF-NZD} is extensively used in the case of ${M=S/I_X}$ where $I_X$ defines a set of points contained in $U_i$.     
\end{remark}

\subsection{The interpolation problem}
Based on our assumptions on the field, we can define the symbolic powers of a point as follows.
\begin{definition} \label{def: symb Power Def}   
Let $p\in\PP$ be a point. The ideal defining the multiple point $mp$ (also termed a fat point ideal or a symbolic power ideal) is
\begin{equation*}
    I_p^{(m)}=\left\{f\in S \mid \frac{\partial^{m-1} f}{\partial x_0^{u_0}\cdots\partial x_n^{u_n}}(p)=0 \text{ where } \sum u_t=m-1 \right\}.
\end{equation*}
Note when the defining ideal of a simple point is a complete intersection the equality $I_p^{(m)}=I_p^m$ holds for all $m \in \N$.
Moreover, for a set $X=\{m_ip_i \}_{i=1}^r$ of multiple points the defining ideal is 
\[I_X=\bigcap_{i=1}^r I_{p_i}^{(m_i)}.
\]
\end{definition}

The ideal above contains the equations of hypersurfaces vanishing to order at least $m_i$ at each point $p_i$. This interpretation leads to the following.

\begin{problem}(Interpolation Problem)
Find the dimension of the vector space of hypersurfaces of a  degree $d$ vanishing to  order at least $m_i$ at each points $p_i$ in a given finite set.
Equivalently, find the Hilbert function the coordinate ring of a set $X=\{m_ip_i \}_{i=1}^r$ of multiple points.
\end{problem}

Using \Cref{def: symb Power Def} for a set $X=\{m_ip_i \}_{i=1}^r$ of points with mixed multiplicity we notice the quantity $\sum_{i=1}^r \binom{n+m_i-1}{n}$ is the  number of conditions imposed on degree $d$ hypersurfaces by the vanishing of all partial derivatives of appropriate order. It is natural to inquire if these conditions are linearly independent. Clearly, the vanishing conditions can not be linearly independent if their number exceeds $s_d$, the dimension of the space of all hypersurfaces of degree $d$. This brings about the next definition which signifies {\it the expected number of linearly independent conditions imposed by $X$.}
\begin{definition}[$\AH_n(d)$ property]\label{def: AH}
Let $X=\{m_ip_i \}_{i=1}^r \subseteq \PP(a_0,\ldots,a_n)$ be a set of $r$ points with defining ideal $I_X\subseteq S$. For $d\in \N$, we say $X$  has {\em property $\AH_n(d)$} or $X$ {\em is $\AH_n(d)$}  if
$$ H_d(R/I_X)=\min \left \{ s_d, \sum_{i=1}^r \binom{n+m_i-1}{n}\right \}. $$
\end{definition}

\Cref{prop: equiv evaluation map} states a well-known equivalent formulations of interpolation in terms of the evaluation map which is now introduced.

\bdf [Evaluation Map] \label{def: EvaluationmMap}
Let $X=\{ m_i p_i \}_{i=1}^r \subseteq \PP(a_0,\ldots,a_n)$ be a set of $r$ double points with $m=\sum_{i=1}^r \binom{n+m_i-1}{n}$. For any $d \in \N$, the {\em evaluation map on $X$ in degree $d$} is the map
\begin{eqnarray*}
    \ev_X : &(S)_d& \to k^{m}\\ 
    &f& \to \left ( \frac{\partial^{m_i-1} f}{\partial x_0^{u_0}\cdots\partial x_n^{u_n}}(p_i)=0 \mid \sum u_t=m_i-1  \right )_{i=1}^r.    
\end{eqnarray*}
\edf

\bpr \label{prop: equiv evaluation map}
For $X \subseteq \PP(a_0,a_1,\ldots,a_n)$ as defined in \Cref{def: EvaluationmMap} and any $d\in \N_0$ the following are equivalent:
\begin{enumerate}
    \item $X$ is $\AH_n(d)$,
    \item $\ev_X:(S)_d \to k^{m}$ has full rank.
\end{enumerate}
\epr 
\begin{remark} \label{rmk: minimum}
    Since $I_X$ is the kernel of the evaluation map and  $S/I_X$ is its image,  it is clear that
    \begin{equation*}
        H_d(S/I_X) \leq \min\{s_d, m \}.
    \end{equation*}
   Therefore if $H_d(S/I_X)$ is equal to $s_d$ or $m$ then that value is automatically the minimum of the two.
\end{remark}



\section{Mixed point interpolation in the weighted projective line}\label{s:line}
This subsection focuses on interpolation in $\PP(a,b)$, where $\gcd(a,b)=1$ and coordinate ring $S=k[z,u]$ where $\deg(z)=a$ and $\deg(u)=b$.

\bpr\label{defid}
Let $I$ be the defining ideal of a point $p=[p_0:p_1]$. Then
\begin{equation*}
I_p=\begin{cases}
(p_1^az^b-p_0^bu^a) & \text{if } p_0,p_1 \neq 0\\
(z) & \text{if }  p_0=0 \text{ and } p_1 \neq 0\\
(u)& \text{if } p_0 \neq 0 \text{ and }p_1 =0.
\end{cases}
\end{equation*}
\epr 

\bpf 
Note $I_p$ is an unmixed height $1$ prime ideal, which implies $I_p$ is principal. Therefore it suffices to find an element of smallest degree in $I_p$.  If $p_0=0$ or $p_1=0$ this  element is evidently $z$ and $u$, respectively. Assuming $p_0,p_1 \neq 0$ implies no pure power of $z$ or $u$ belongs to $I_X$. Because $\gcd(a,b)=1$ and $ (p_1^az^b-p_0^bu^a) \subseteq I_X$, the least degree homogenous element in $I_X$ is $p_1^az^b-p_0^bu^a$.
\epf

\bpr \label{GK}
A set of general points with mixed multiplicity, $X=\{r_ip_i \}_{i=1}^n\subseteq \PP(a,b)$, is $\AH_1(d)$ for all $d\in \N$, that is
\begin{equation} \label{eq: AH_1(d)}
     H_d(S/I_X)=\min \{ s_d, \sum_{i=1}^n r_i\}. 
\end{equation}

\epr 
\bpf
It suffices to demonstrate one such set. (See \Cref{rmk: specilization}.) Let $X=\{r_1p_1,\ldots, r_np_n\}$ where $p_i=[p_i^{(0)}:p_i^{(1)}]$ and $[0:p_1] \not \in X$.
 Recall that $s_d =H_d(k[z,u])$ is the number of solutions to 
\begin{equation} \label{eq: proj line}
     ax+by=d \text{ where } x,y\in \N_0.
\end{equation}
We wish to interpret $H_d(S/I_X)$ in terms of \Cref{eq: proj line}.
For any $i$, \Cref{defid} implies $I_{p_i}=((p_i^{(0)})^bz^b-((p_i^{(1)})^au^a)$ is a principal prime ideal and thus
\begin{eqnarray*}
I_X &=& \bigcap_{i=1}^n I_{p_i}^{(r_i)}=\bigcap_{i=1}^n \left( (p_i^{(0)})^bz^b-((p_i^{(1)})^au^a)^{(r_i)} \right)\\
&\underset{=}{*}& \left(\prod_{i=1}^n ((p_i^{(0)})^bz^b-((p_i^{(1)})^au^a)^{r_i} \right).
\end{eqnarray*}
Identity * follows since each ideal is a complete intersection. 
Replacing $I_X$ by its initial ideal $(u^{ar})$, where $r=\sum_{i=1}^n r_i$, we notice that 
\[
S/I_X =\text{span}_k \{z^x u^y\; | \; 0\leq y \leq ar-1, \; 0 \leq x \} \underset{k-\text{vector space}}{\cong} k[z] \otimes_k k[u]/(u^{ar}).
\]
Hence $H_d(S/I_X)$ is the number of solutions to \Cref{eq: proj line} \textit{ with the restriction} that $0\leq y \leq ar-1$.
We will show
\begin{equation} \label{eq: line computation}
   H_d(S/I_X)=\begin{cases}
        s_d &\text{ if } d<b(ar-1)\\
        r   &\text{ if } d\geq b(ar-1).
    \end{cases}
\end{equation}
Suppose $d < b(ar-1)$. Then $ 0\leq y \leq ar-2$ regardless of whether we are computing $s_d$ or $H_d(S/I_X)$, which implies that $H_d(S/I_X)=s_d$ holds in this case.
Let $d \geq b(ar-1)$. Then $y=ar-i$ for $1 \leq i \leq ar$. It suffices to count how many values of $i$ yield a solution. From \Cref{eq: proj line} we get
\[
ax+b(ar-i)=d \Rightarrow bi+d \equiv 0 \pmod a \Rightarrow (-b)^{-1}d\equiv i \pmod a
\]
Hence $[i]_{a}=[(-b)^{-1}d]_a$ and $1\leq i \leq ar$. Therefore $H_d(S/I_X)$ is equal to the number of distinct coset representatives of $[(-b)^{-1}(d)]_a$ in the interval $[1,ar]$; of which there are $r$. Hence $H_d(S/I_X)=r$ for $d \geq b(ar-1)$.
Combining \Cref{rmk: minimum} and \Cref{eq: line computation} implies \Cref{eq: AH_1(d)}
\epf

\section{Simple point interpolation in the \wps} \label{section: interpolation in \wps}
The objective of this section is to establish simple point interpolation in $\PP(a_0,\ldots,a_n)$. \Cref{prop: simple point interpolation} is implicitly used in subsequent proofs of \cref{2POINTS}, \Cref{AH123} and \Cref{Theorem: 1bc bound}. Our proof is inspired by the one in \cite{JackHuizenga}. 

\bthm \label{prop: simple point interpolation}
Let $\PP=\PP(a_0,a_1,\ldots,a_n)$ and $X=\{p_i \}_{i=1}^r$ be a set of general simple points. $X$ is $\AH_n(d)$ for all $d$, that is
\begin{equation} \label{eq: simple pt interpolation}
    H_d(S/I_X)=\min \{ s_d, r\}
\end{equation}
\ethm

\bpf 
Let $k$ be an algebraically closed field. 
We use induction on $r$, the number of points in $X$.
Let $r=1$ and choose $p=[p^{(0)}:\cdots:p^{(n)}]$ such that $p^{(j)} \neq 0$ for all $j$. By \Cref{lem: monomial curve}, $S/I_p\cong k[t^{a_j} \mid 0\leq j \leq n]$. The observations below implies \Cref{eq: simple pt interpolation} for $X=\{p\}$ where $A=\{a_j \}_{j=0}^n$ and $\mathcal{N}(A)$ denotes the numerical semigroup generated $A$.
\begin{equation*}
    H_d(S/I_X)=\begin{cases}
        1 &\text{ if } d\in \mathcal{N}(A)\\
        0 &\text{else}
    \end{cases},
\; \; 
s_d=\begin{cases}
        \geq 1 &\text{ if } d\in \mathcal{N}(A)\\
        0 &\text{else}
    \end{cases}
\end{equation*}

Let $X'$ be a set of size $r-1$ for which \Cref{eq: simple pt interpolation} holds for all $d$. We claim for any fixed $D$ and some choice of a point $p$, there exists a set $X=X'\cup \{p\}$ such that \Cref{eq: simple pt interpolation} holds for $H_D(S/I_X)$. If $s_D \leq r-1$, choose $p \not \in X'$. Then
\[(I_X)_D \subseteq (I_{X'})_D =0 \Rightarrow H_D(S/I_X)=s_D < r. \]
If $r-1 < s_D$ then there exists $f \in (I_{X'})_D$ such that $f \neq 0$. Choose $p \in\PP$ such that $p \not \in X' \cup V(f)$. Such choice is possible because $X\cup V(f)\neq \PP$ by the Nullstellensatz. In particular, $f(p) \neq 0$.
Observe that $(I_X)_D = \{g \in (I_{X'})_D\; |\; g(p)=0 \}$ is the kernel of the map
\begin{equation*}
    (I_{X'})_D \xrightarrow{} k \text{ where }  g \mapsto g(p),
\end{equation*}
which is surjective since $f(p)\neq 0$. Hence $H_D(I_X)=H_D(I_{X'})-1$ and
\begin{equation*}
    H_D(S/I_X)=H_D(S)-H_D(I_X)=H_D(S)-H_D(I_{X'})+1=H_D(S/I_{X'})+1=r, 
\end{equation*}
where $r=\min\{ s_D,r\}$ by choice of $D$. This proves the claim.

Because the construction of the set $X$ depended on the degree $D$, the claim above shows for a {\it fixed D} there exists a set of size $r$ that satisfies \Cref{eq: simple pt interpolation}. Equivalently, there exists a non-empty open set in the parameter space $\PP(a_0,\ldots,a_n)^{r}$, denoted $C_D$, whose points satisfy \Cref{eq: simple pt interpolation}. Next we show the existence of a single set $Y$ of size $r$, independent of $D$, that satisfies \Cref{eq: simple pt interpolation} for all degrees $d$.

We argue to obtain an integer $t$ such that $r\leq s_d$ for all $d\geq t-a_n$. The well-formedness of $a_i$'s, \Cref{rmk: well-formed}, implies $\gcd(a_0,a_1,\ldots,a_n)=1$. Since $r$ is fixed and the Hilbert function of $S$ for $d\gg0$ is given by a quasi-polynomial 
\begin{equation*}
    H_d(S)=
        P_i(d) \text{ for } d\equiv i \pmod{\lcm(a_0,\ldots,a_n)}
\end{equation*}
where each $P_i$ is a polynomial of degree $n\geq 1$ such that $P_i(d)>0$ for $d\gg 0$, it follows that each sequence $\{P_i(d)\}_d$ is eventually increasing. Thus the existence of the desired $t$ is established. Moreover choose the smallest such integer $t$.
Consider
\begin{equation*}
    C= \underset{0\leq D \leq t+\cdot a_n}{\bigcap}C_D \cap \left(\PP(a_0,\ldots,a_n)^r\setminus Z\right ),
\end{equation*} 
where $Z$ represents the closed set of $r$-tuples of points in $\PP(a_0,\ldots,a_n)$ where each coordinate hyperplane $V(x_j)$ contains at least one of the points or where any two or more of the points coincide.

By construction $C$ is a non-empty open set since it is the intersection of finitely many non-empty open sets and points of $C$ satisfy \Cref{eq: simple pt interpolation} for all $0\leq D \leq t+a_n$. 
We finish the proof by showing this implies points of $C$ satisfy \Cref{eq: simple pt interpolation} for all $d$.
Fix an $r$-tuple of $C$ which corresponds to a set $Y\subseteq \PP$ of $r$ points. Notice some variable $x_j$ is a non-zero divisor on $S/I_{Y}$ and \Cref{equation: HF-NZD} is now applicable. 
Next we show 
\begin{equation} \label{eq: cyclic}
           H_d(S/(I_{Y}+x_j))=0 \text{ for } t\leq d\leq t+ a_n.
\end{equation}
Let $d=t+\gamma$ where $0\leq \gamma \leq a_n$. Note $t-a_n \leq t-a_j$ and $H_{d'}(S/I_Y)=r$ for all $t-a_n \leq d' \leq t+a_n$ by choice of $t$. Plugging $d$ in \Cref{equation: HF-NZD} gives 
\begin{equation*}
      H_{t+\gamma}(S/I_{Y})=H_{t+\gamma-a_j}(S/I_{Y})+H_{t+\gamma}(S/(I_{Y}+x_j)). 
\end{equation*}
Which implies \Cref{eq: cyclic}. Finally, $S/(I_{Y}+x_n)$ is a graded cyclic $S$-module and since $[S/(I_{Y}+x_n)]_d\subseteq [S/(I_{Y}+x_n)]_{d-a_n}S$ for all $d\geq a_n$, \Cref{eq: cyclic} yields
 \begin{equation*}
            H_d(S/(I_{Y}+x_j))=0 \text{ for } t+ a_n\leq d.
\end{equation*}
For an arbitrary choice of $Y$ we have shown
\begin{equation}\label{eq: HF simple pts}
    H_d(S/I_{Y})=\begin{cases}
         r & d\geq t-a_n\\
        s_d & d<t-a_n
    \end{cases}
    \overset{*}{=}\min\{r,s_d \} \text{ for all }d.
\end{equation}
Where $*$ follows from \Cref{rmk: minimum}. Thus a set of $r$ general simple points satisfies \Cref{eq: simple pt interpolation}. 
\epf

\section{Double point interpolation in the weighted projective plane} \label{section: weighted projective plane}

This section focuses on interpolation in weighted projective planes of the form $\PP(1,b,c)$, with coordinate ring $k[z,u,v]$ with $\deg(z)=1$, $\deg(u)=b$ and $\deg(v)=c$. We work with the evaluation map to show general double points have the $\AH_2(d)$ property for all $d$, which recall is to show $$H_d(S/I_X)=\min \{s_d, 3\cdot |X| \}.$$

\begin{remark} \label{rmk: specilization}
It is well-known that in order to prove a general set of points has the $\AH_n(d)$ property, it suffices to demonstrate one such set. One can further limit the size of this set to at most two values. For a fixed $n$ and $d$, it suffices to prove the property $\AH_n(d)$ for one set of double points $X$, where $|X|=r$ and $r=\LF \frac{1}{n+1} s_d \RF$ and/or $r=\LC \frac{1}{n+1} s_d \RC$; see \cite[Corrollary D.4-D.5]{HM21}. We employ this in the weighted projective plane where $n=2$.    
\end{remark} 
\begin{remark} \label{remark: Weight of 1}
The existence of a weight of $1$ ensures the following:
\begin{enumerate}[leftmargin=2em]
    \item The defining ideal of a point is a complete intersection. (\Cref{POINTS ARE CI}.)
    \item We can choose $U_0=\PP(1,b,c)\setminus V(z)$ as our locus of interpolation. Note $U_0$ is a non-empty open dense set in the parameter space, in this case $\PP(1,b,c)$. Hence our notion of general refers to sets of points that are contained in $U_0$ and is made more precise in different results. The advantage is if $X\subseteq U_0$ is a set of points and $H_d(S/I_X)=(n+1)r$ for some $d$, then \cref{def: EvaluationmMap} and \Cref{equation: HF-NZD} imply $H_{d'}(S/I_X)=(n+1)r$ for all $d'\geq d$. 
    \item \thickmuskip=0mu The coordinate ring of $\PP(1,b,c)$, $S=k[z,u,v]$, has non-decreasing Hilbert function. Since $z$ is a degree $1$ non-zero divisor, \Cref{equation: HF-NZD} yields $H_d(S)\geq H_{d-i}(S)$ for all $i$ hence $\{s_i\}$ forms a non-decreasing sequence. 
\end{enumerate}    
\end{remark}
 \Cref{NECK} shows to find a space  in which $r$ general double points are $\AH_2(d)$ for all $d$, it is necessary for $c$ to not be too large compared to $b$ and $r$. \Cref{mon3k} provides a numerical observation towards this end.
 
\blm \label{mon3k}
Fix positive integers $b,c$ and $r$ such that $rb<c<(r+1)b$. For the coordinate ring $S$ of $\PP(1,b,c)$, the smallest $d$ such that $s_d=3r$  is $d=(r-1)b+c$.
\elm

\bpf 
For $d=(r-1)b+c$, observe that $\deg(v^2)>d$ and we get
\begin{equation} \label{eq: V^2}
     k[z,u,v]_{d}=k[z,u]_{d}\oplus v\cdot k[z,u]_{d-c}.
\end{equation}
This implies by \Cref{eq: sd for 1b} the claimed identity
\begin{equation*}
s_d=\left \lfloor \frac{d}{b} \right \rfloor +1 + \left \lfloor \frac{d-c}{b} \right \rfloor +1=3r.
\end{equation*}
To see the minimality of $d$, note for $d'<d$, an analogous decomposition to \Cref{eq: V^2} implies
\begin{equation*}
s_{d'}=\left \lfloor \frac{d'}{b} \right \rfloor +1 + \left \lfloor \frac{d'-c}{b} \right \rfloor +1<3r. \qedhere
\end{equation*}
\epf

\bpr \label{NECK}
If $r$ general double points are $\AH_2(d)$ for all $d$ in $\PP(1,b,c)$, then 
\begin{eqnarray*}    
    c &\leq& (r+1) \text{ for } b=1\\    
    c &<& (r+1)b \text{ for } b\neq 1  
\end{eqnarray*}
\epr

\bpf
Let $X=\{2 p_i\}_{i=1}^r$ be a set of points such that $p_i\neq[0:1:0]$ Assume for contradiction that $c>(r+1)b$.
It suffices to show $s_{2rb}\leq 3r$. This claim implies $(I_X)_{2rb}=0$. But by \Cref{ptsplne}, the square of the ideal of each point, $I_{p_i}^2$, contains a generator, $f_i$, of degree $2b$. Clearly $ F=\prod_{i=1}^r f_i \in I_X $ and $\deg(F)=2rb$. This contradicts $X$ being $\AH_2(2rb)$.

To verify $s_{2rb}\leq 3r$, observe $\deg(v^2)=2c>2rb$ and \Cref{eq: V^2} yields
\begin{equation*}
    s_{2rb}=\LF \frac{2rb}{b} \RF +1 + \LF \frac{2rb-c}{b} \RF +1\leq 2r+1+(r-2)+1\leq 3r. \qedhere
\end{equation*}
\epf 
\subsection{Small number of points in $\PP(1,b,c)$} \label{subsection: small number of double points}
\Cref{remark: one degree} explains our proof strategy for the rest of \cref{section: weighted projective plane}. 
\begin{remark} \label{remark: one degree}
Let $M_d$ denote the matrix of $\ev_X:(S)_d\to k^{3r}$. We wish to demonstrate one set of double points $X$, of size $r$, that is $\AH_2(d)$ for all $d$. If  there is a degree $d$ such that $s_d=3r$, then it suffices to check the rank of $M_d$ is maximal for smallest such $d$. By \Cref{remark: Weight of 1} and \Cref{prop: equiv evaluation map}, it follows that $X$ is $\AH_2(d')$ for all $d'\geq d$. Moreover, for $d''\leq d$ having a weight of $1$ implies
\begin{equation*}
      (I_X)_{d} = 0 \Rightarrow (I_X)_{d''}=0.
\end{equation*}
Combining this with \Cref{remark: Weight of 1} -- part 3
implies $X$ is $\AH_2(d'')$ for $d'' \leq d$. For certain numbers $r$  the degree where $s_d=3r$ is identified in \Cref{mon3k}.  

We show $M_d$ has full rank by exhibiting a set of points at which $\det(M_d)$, a polynomial in $2r$ variables corresponding to the coordinates of $r$ points, is not zero. Thus $\det(M_d)$ is not the zero polynomial. Hence the notion of general will refer to points whose first coordinate is nonzero and last two coordinates are outside of $V(\det(M_d))\subseteq \PP^{2r}$. 
\end{remark}
In \Cref{1POINT} and \Cref{2POINTS} we give necessary and sufficient condition for $1$ point and $2$ points to be $\AH_2(d)$ for all $d$.

\bpr \label{1POINT}
One general double point in $\PP(1,b,c)$ is $\AH_2(d)$ for all $d$ if and only if
\begin{enumerate}
    \item $c=2$ when $b=1$,
    \item $b<c<2b$ when $b\neq 1$.
\end{enumerate} 
\epr 

\bpf 
$(\Rightarrow)$ This is a direct consequence of \Cref{rmk: well-formed} and \Cref{NECK}.

\noindent $(\Leftarrow)$ 
Let $X=\{2p\}$ with $p\not\in V(z)$. When $b=1$ and $c=2$ observe that $H_2(S/I_X)=3$ because $s_2=4$ and $H_2(I_X)=1$. Further $H_1(S/I_X)=1$. By a similar reasoning to \Cref{remark: one degree} $X$ is $\AH_2(d)$ for all $d$. Suppose $b<c<2b$. By \Cref{mon3k} and \cref{remark: one degree}, it suffices to check $\AH_2(d)$ for $X$ in degree $d=c$. By \Cref{ptsplne}, the smallest degree of a generator in $I_X$ is $2b>c$. Hence
$  H_{c}(S/I_X)=3=\min\{s_c, 3 \}$.
\epf
Our results so far have been characteristic free. In \Cref{2POINTS} to ensure certain determinants are non-zero we make assumptions on $\Char(k)$.
\bpr \label{2POINTS}
Let $k$ be an infinite field with $\Char(k) \not \in \{ 2,3,b,b+c\}$. Two general double points in $\PP(1,b,c)$, where $(1,b,c)\neq(1,2,3)$, are $\AH_2(d)$ for all $d$ if and only if 
\begin{enumerate}
    \item $b=1$ and $c=2$ or $c=3$,
    \item $2b<c<3b$, or
    \item $\frac{3}{2}b<c<2b$
\end{enumerate}
\epr 
\bpf 
$(\Leftarrow)$
If $b=1$, $c=3$ and $X$ is a set of two double points, one can use the defining equations of the points to show $H_3(S/I_X)=5=s_3$ and $H_4(S/I_X)=6$. By a similar reasoning to \Cref{remark: one degree}, $X$ is $\AH_2(d)$ for all $d$.
Next let $2b<c<3b$ or b=1 and c=2. \Cref{mon3k} and direct computations imply $s_{b+c}=6$ and this is the smallest degree in which $S$ has $6$ monomials: $z^{b+c}, z^cu,z^{c-b}u^2, z^{c-2b}u^3, z^bv, uv$. The evaluation matrix with respect to this basis for points with coordinates $[1:p_1:p_2]$ and $[1:q_1:q_2]$ is

\begin{center}
$M_{b+c}=\left[\begin{matrix}
 b+c&cp_1&(c-b)(p_1)^{2}&(c-2b)\,(p_1)^{3}&b\,p_2&0\\
     0&1&2\,p_1&3\,(p_1)^{2}&0&p_2\\     b+c&cq_1&(c-b)(q_1)^{2}&(c-2b)\,(q_1)^{3}&b\,q_2&0\\
     0&1&2\,q_1&3\,(q_1)^{2}&0&q_2\\     
     0&0&0&0&1&p_1\\
     0&0&0&0&1&q_1\\
     \end{matrix}\right]$

\end{center}    
with determinant $\det(M_{b+c})=(-1)(p_1-q_1)^5(b+c)^2\neq 0$.

Finally suppose $\frac{3}{2}b<c<2b$. Adapting the proof of \Cref{mon3k}, one can show $d=3b$ is the first degree where $s_d=6$. The evaluation matrix with respect to the basis $z^{3b}, z^{2b}u,z^{b}u^2, u^3,  z^{2b-c}uv, z^{3b-c}v$ and the points fixed before is
\begin{center}
$M_{{3b}}=\left [\begin{matrix}
     3\,b&2\,b\,p_1&b\,(p_1)^{2}&(p_1)^{3}&(2b-c)\,p_1\,p_2&(3b-c)\,p_2\\
     0&1&2\,p_1&3\,(p_1)^{2}&p_2&0\\
     0&0&0&0&p_1&1\\
     3\,b&2\,b\,q_1&b\,(q_1)^{2}&(q_1)^{3}&(2b-c)\,q_1\,q_2&(3b-c)\,q_2\\
     0&1&2\,q_1&3\,(q_1)^{2}&q_2&0\\
     0&0&0&0&q_1&1\\
     \end{matrix}\right]$
\end{center}
where $\det(M_{3b})= 3b (p_1-q_1)^{3}\left[3\,b\,(p_1-q_1)^{2}-2\,p_1^{2}-2\,p_1\,q_1-2\,q_1^{2}\right]$.

$(\Rightarrow)$ Suppose a set of two general points, $X$, is $\AH_2(d)$. General in this context refers to $X$ not containing any of the points $[0:1:0]$ or $[1:0:0]$. When $b$ is $1$ the result is immediate from \Cref{NECK} so suppose $b\neq 1$
\Cref{NECK} results in $c< 3b$ and if $2b<c<3b$ we are done. Suppose $b<c<2b$. We wish to show $\frac{3}{2}b<c$. The equality $3b=2c$ is impossible by well-formedness and the hypothesis which excludes the possibility $b=2, c=3$. Therefore one may assume for contradiction that $2c<3b$. Direct calculations show $s_{2c}=6$. By assumption 
\begin{equation*}
    H_{2c}(S/I_X)=\min\{6,s_{2c} \}=6.
\end{equation*}
This implies $(I_X)_{2c}=0$. We arrive at a contradiction by constructing a non-zero curve in $(I_X)_{2c}$. Since $b<c<2b$ it follows $s_c=3$ and thus there is a degree $c$ curve with defining equation $F$ passing through the two points. Since $F^2\in [I_X]_{2c}$ we have reached the desired contradiction.
\epf 

\subsection{A sufficient condition for $r$ points} \label{subsection: Sufficient k points}
The main goal of this subsection is to give a sufficient condition for the weights so that $r$ general double points in $\PP(1,b,c)$ are $\AH_2(d)$ for all $d$. The proof of \Cref{suffkpt} is a generalization of case 1 in \Cref{2POINTS}.
\bpr \label{suffkpt}
Let $r \in \N$. If $rb<c<(r+1)b$ then a set of $r$ general double points in $\PP(1,b,c)$ is $\AH_2(d)$ for all $d$.
\epr 

\bpf 
Let $X=\{2p_i\}$, where the points $p_i=[1:p_i^{(1)}:p_i^{(2)}]$ for $1\leq i\leq r$ satisfy the following conditions:
 \begin{enumerate}
     \item $p_i^{(1)}\neq 0$ and $  p_i^{(2)} \neq 0 \text{ for all } i$,
     \item $p_i^{(1)} \neq p_{i'}^{(1)}$ for all $i\neq i'$.
 \end{enumerate}
Let $c=rb+j$ and $d=(r-1)b+c=(2r-1)b+j$ which, by \Cref{mon3k}, is the smallest degree $d$ where $s_d=3r$. By \Cref{remark: one degree} it suffices to show $H_{d}(S/I_X)=3r$. 

We order the monomials of degree $d$ as follows
\begin{equation*}
 z^{((2r-1)-s)b+j}u^{s}, z^{((r-1)-t)b}u^tv  \text{ where } 0\leq s\leq 2r-1, \;  \; 0\leq t \leq r-1.    
\end{equation*}   
The rows of the evaluation map correspond to $\frac{\partial}{\partial z}|_{p_i}, \frac{\partial}{\partial u}|_{p_i}$ and $\frac{\partial}{\partial v}|_{p_i}$ and the columns to monomials as ordered.  Based on our ordering of the monomials, the first $2r$ monomials do not contain $v$ and the last $r$ do. 



The matrix $M_d$ of the evaluation map with rows indexed by $\frac{\partial}{\partial z}|_{p_i}, \frac{\partial}{\partial u}|_{p_i}$ in the first block and $\frac{\partial}{\partial v}|_{p_i}$ in the second block becomes

\begin{equation*}
\displaystyle
\arraycolsep=2.5pt 
\left[
    \begin{array}{c|c}
        \begin{array}{cccc}
        (2k-1)b+j & ((2k-2)b+j)p_1^{(1)} & \cdots &(p_1^{(1)})^{2k-1}\\
        \vdots & \vdots && \vdots  \\
        (2k-1)b+j & ((2k-2)b+j)p_k^{(1)} & \cdots & j(p_k^{(1)})^{2k-1}
        \end{array}   &  ** \\ 
        \hline
         \bigzero & 
         \begin{array}{ccccc}
        1 & p_1^{(1)} & (p_1^{(1)})^2 & \ldots & (p_1^{(1)})^{k-1}\\
        \vdots  & \vdots & \vdots && \vdots\\
        1 & p_k^{(1)} & (p_k^{(1)})^2 &\ldots &  (p_k^{(1)})^{k-1}    
        \end{array}
    \end{array}
\right]
\end{equation*}
The $2r\times 2r$ minor in the upper left quadrant corresponds to the evaluation map in degree $d$ for the points $[1:p_i^{(1)}]\in\mathbb{P}(1,b)$. Assumption $(1)$ allows us to view the points in $\PP(1,b)$ and this minor is non-zero by \Cref{GK}. The bottom right block is the $r\times r$ Vandermonde minor which is non-zero by assumption $(2)$. Hence $M_d$ has full rank which completes the proof.
\epf

\section{The Alexander--Hirschowitz theorem in $\PP(1,2,3)$}\label{123}
\noindent This section is dedicated to the proof of the following result. 
\bthm[Alexander Hirschowitz theorem in $\PP(1,2,3)$]\label{AH123} 
Let $k$ be a field where $\Char(k) \not \in \{2,3,5 \}$.
Let $X \subseteq \PP(1,2,3)$ be a set of $r$ general double points with defining ideal $I_X$. Then $X$ is $\AH_2(d)$ for all $d$ with no exceptions. That is 
\begin{equation*}
    H_{d}(S/I_X)=\min \{s_d, 3r \} \text{ for all }d.
\end{equation*}
\ethm

\bpf 
We use strong induction on $d$. By \Cref{rmk: specilization}, for each $d$ and for each of the values $r=\LF \frac{1}{3} s_d \RF$ and $r=\LC \frac{1}{3} s_d \RC$ it suffices to demonstrate one set of size $r$ that is $\AH_2(d)$.
\Cref{IND6} proves $6$ base cases for the induction, that is any finite set of \gdp is $\AH_2(d)$ for $0\leq d\leq 5$. 

For any $d \geq 6$, and either value of $r$, we satisfy the hypothesis of \Cref{TERA} which constructs a set of size $r$ that is $\AH_2(d)$. This entails showing:
\begin{itemize}[leftmargin=1em]
    \item The numerical condition of \Cref{TERA}, \Cref{eq: TERRACINI CONDITION}, is satisfied by \Cref{TERANUM}. Let $q$ denote the positive integer mentioned in this condition.
    \item Part~(1) of \Cref{TERA} is the content of \Cref{GK} when $r_i=2$ for all $i$.
    \item Part~(2) of \cref{TERA} follows because by the induction hypothesis a set of general double points of any size in $\PP(1,2,3)$ is $\AH_2(t)$ for $1\leq t \leq d-1$. In particular, a set of $r-q$ general double points in $\PP(1,2,3)$  is $\AH_2(d-i)$ and $\AH_2(d-2i)$ for $i\in \{1,2,3\}$. In  \Cref{lem: Chandler} we show, using \Cref{C13} and the previous considerations, that adding $q$ general simple points lying on a projective line, to the set of $r-q$ \gdp  preserves the $\AH_2(d-i)$ property. This constructs the desired set in part~(2) of \Cref{TERA}. \hfill $\qedhere$  
\end{itemize}
\epf 
\Cref{TERA} is the main inductive step used in \Cref{AH123} and in the standard graded case is proven in \cite{HM21} and \cite{BO08}. The assumption on the ambient space in \Cref{TERA} allows us to conclude the defining ideal of a simple point is a complete intersection (\Cref{POINTS ARE CI}, \Cref{defidplane}), a fact which is fundamental to the proof presented in \cite{HM21}. In particular when $n\geq 3$  all general points in $\PP$ are considered in $U= \bigcup_{\{ j\; \mid\;  a_j=1\} } U_j$ to ensure simple points are complete intersections.

\bthm[Generalized Terracini's Inductive argument] \label{TERA}
Consider $\PP=\PP(1,\ldots,a_n)$, with coordinate ring $S=k[x_0,\ldots,x_n]$ where $\deg(x_i)=a_i$.  Assume $1\leq q \leq r$ and $d\in Z_+$ satisfy one of the following inequalities for at least one $i$ in the range $0\leq i\leq n$,
\begin{equation} \label{eq: TERRACINI CONDITION} 
    (n+1)r-s_{d-a_i} \leq nq \leq \overline{s}_d \qquad \text{ or } \qquad  \overline{s}_d \leq nq \leq (n+1)r-s_{d-a_i},
\end{equation}
where $\overline{S}=S/(x_i)$ is the coordinate ring of the hyperplane $L=V(x_i)\cong \PP(a_0,\ldots,\hat{a_i},\ldots,a_n)$ and $\overline{s}_d=\dim_k(\overline{S}_d)$.
If
\begin{enumerate} 
    \item a set of $q$ general double points in $L$ is $\AH_{n-1}(d)$ and
    \item the union of a set of $r-q$ general double points in $\PP$ and a set of $q$ general simple points in $L$ is $\AH_n(d-a_i)$
\end{enumerate}
then a set of $r$ general double points in $\PP$ is $\AH_n(d)$.
\ethm
\bpf 
See \cref{Appendix: 123}.
\epf
It is notable that Terracini originally came up with \Cref{eq: TERRACINI CONDITION} to study configurations of points in $\PP^3_k$. For fixed $d$ and $r$ we refer to the existence of $q$ that satisfies one of the inequalities in \Cref{eq: TERRACINI CONDITION} as ``the numerical condition'' of \Cref{TERA}. In the straight projective space there are infinitely many values of $r$ and $d$ where this condition fails. See \cite[Section 4]{BO08}. Even in $\PP^2_k$ this numerical condition fails for $d=2,4$ and $r=2, 5$, respectively, which are exactly the exceptions in \Cref{thm: AH}.

In \Cref{TERANUM} we show this numerical condition holds for all $d$ in the case of $\PP(1,2,3)$ and based on \cref{thm: 123 ONLY} it is the only such space. The main difference with $\PP^2_k$ is the different numerics of projective lines in $\PP(1,2,3)$.

Part (2) of the hypothesis of \Cref{TERA} is satisfied with the aid of \Cref{C13} which is proved in \Cref{Appendix: 123}. Originally due to K. Chandler \cite[Lemma 3]{Chandler01}, this lemma is modified to our setting and shows if a certain numerical criterion is satisfied, then one can add simple points all lying on a hyperplane to a set of points that are $\AH_n(d)$ while preserving the property. This is made precise in \Cref{C13}.

\blm[Chandler's Lemma] \label{C13}
Let $I$ be a saturated homogeneous ideal in $S=k[x_0,\ldots,x_n]$ with $\deg(x_i)=a_i$, the homogeneous coordinate ring of $\PP(a_0,\ldots,a_n)$. Assume $l$ is regular on $S/I$. Let $\overline{S}=S/(l)$ and $\deg(l)=i$. Fix $d \in Z_+$. Then following are equivalent:
\begin{enumerate}
    \item A general set $Y_0$ of $q$ reduced points in $V(l)$ satisfies
    $$H_{d}(S/(I \cap J_{Y_0}))=H_{d}(S/I)+q $$

    \item $H_{d}(S/I)+q \leq H_{d-i}(S/I)+H_{d}(\overline{S})$.
\end{enumerate}
\elm


Next we prove $6$ base cases of induction used in proof of \Cref{AH123}.

\bpr\label{IND6}
Let $k$ be an algebraically closed field such that $\Char(k) \not \in \{2,3, 5 \}$. Let $X$ be a set of $r$ double points with $I_X$ as its defining ideal. Assume no point of $X$ lies on $V(z)$. Further assume for distinct points $p_i$ and $p_j$ in $X$ where $p_i=[p_i^{(1)}:p_i^{(2)}:p_i^{(3)}]$ and $p_j=[p_j^{(1)}:p_j^{(2)}:p_j^{(3)}]$, $p_i\symb \neq p_j\symb $ for all $i\neq j$. Then $X$ is $\AH_2(d)$ for all $0 \leq d \leq 5$.
\epr 
\bpf 
By \Cref{rmk: specilization} it suffices to consider $r=\LF \frac{1}{3} s_d \RF$ and/or $r= \LC \frac{1}{3} s_d \RC$. Note that for $d\leq 5$ we have $s_d\leq 5$ and thus $r\leq 2$. Anytime $r=1$, the result follows from \Cref{1POINT}, whose proof requires the points not to lie on $V(z)$. For $d=4$ or $d=5$ and $r=2$, the result follows from \Cref{2POINTS}, whose proof requires  $p_i\symb \neq p_j\symb$.
\epf

We address a subtly in our notion of general and the construction in \Cref{TERA}. The condition that no point of $X$ lies on $V(z)$ in \Cref{IND6} does not obstruct the procedure in \Cref{TERA}, which requires specializing $q$ points to a hyperplane, $L$. These $q$ points need to be general {\em in $L$}; but neither condition $(1)$ or $(2)$ requires $q$ general points in $\PP(1,2,3)$. 

Let $d \geq 6$. Assume any finite set of general double points $X$ 
is $\AH_2(t)$ for $0\leq t \leq d-1$. We will use \Cref{TERA}, to show for $r=\LF \frac{1}{3}s_d \RF$ and $r=\LC \frac{1}{3}s_d \RC$ that there exists a set of $r$ double points that are $\AH_2(d)$. This inductive procedure for $d=14$ is demonstrated in \Cref{example: Terra step}, after which the rest of this section is dedicated to establishing the numerical condition in \Cref{TERANUM} and condition $(2)$ in \Cref{lem: Chandler}. To study \Cref{eq: TERRACINI CONDITION} in the context of the $\PP(1,2,3)$, we introduce the following notation.

\begin{notation}\label{not: s'}
 Let  $S=k[z,u,v]$ denote the coordinate ring of $\PP(1,2,3)$ where $\deg(z)=1$, $\deg(u)=2$ and $\deg(v)=3$ and recall that $s_d = \dim_k(S_d)$. Let $S'=k[u,v]$, $S''=k[z,v]$ and $S'''=k[z,u]$ denote the coordinate rings of the  projective lines $L_1\cong \PP(2,3)$, $L_2\cong \PP(1,3)$, and $L_3 \cong \PP(2,3)$ respectively, where $L_1=V(z), L_2=V(u), L_3=V(v)$ and denote the Hilbert function for each of these rings in degree $d$ by $s'_d, s''_d, s'''_d$, respectively. 
 \end{notation}
These values can be computed explicitly.

\begin{remark}\label{FORMULA123}
For any $d \in \N_0$ we have 
\begin{enumerate}
    \item $s_d=\lfloor d^2/12+d/2+1 \rfloor$ according to \cite[Page 547]{stanley1}.    
    \item $s'''_d=\lfloor d/2 \rfloor +1$ cf.~\Cref{eq: sd for 1b}.
    \item $s''_d=\lfloor d/3 \rfloor +1$ cf.~\Cref{eq: sd for 1b}.
    \item
   $
    s'_d = 
    \begin{cases}
          \LF d/6  \RF   &d \equiv 1 \pmod 6   \\   
        \LF d/6 \RF +1 &\text{otherwise}.
    \end{cases}
$ cf.~ \Cref{eq: sd for ab}.
\end{enumerate} 
\end{remark}

\bex \label{example: Terra step}
We exhibit Terracini's technique for $d=14$. Since $s_{14}=24$, by \Cref{rmk: specilization} it suffices to show $8$ double points are $\AH_2(14)$. Checking the numerical condition (\Cref{eq: TERRA SET 1} and \Cref{eq: TERRA SET 2}) the inequality $3(8)-s_{11}=8\leq s'''_{14}=8$ is the only one for which $q$ exists; this corresponds to specializing $4$ double points to $\PP(1,2)$.

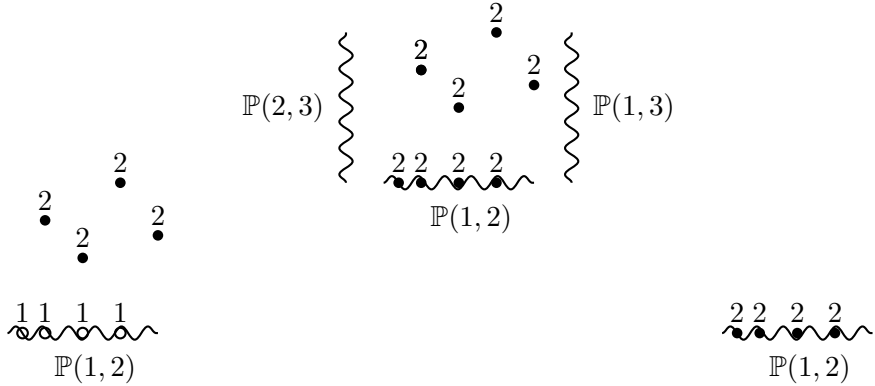
\begin{center}
\begin{tikzpicture}[thick]
\fill (7,0) circle (2pt) node[above] {2};
\fill (6.5,.5) circle (2pt) node[above] {2};
\fill (7.5,1) circle (2pt) node[above] {2};
\fill (6.5,.5) circle (2pt) node[above] {2};
\fill (8,.3) circle (2pt) node[above] {2};
\fill (6.2,-1) circle (2pt) node[above] {2};
\fill (6.5,-1) circle (2pt) node[above] {2};
\fill (7,-1) circle (2pt) node[above] {2};
\fill (7.5,-1) circle (2pt) node[above] {2};
\path [draw=,snake it, ] (6,-1) -- (8,-1) node [label=below left:{$\PP(1,2)$}] {};
\path [draw=,snake it]  (5.5,-1) -- (5.5,1)  node [midway, label= left :{$\PP(2,3)$}] {};;
\path [draw=,snake it]  (8.5,-1) -- (8.5,1)  node [midway, label=right:{$\PP(1,3)$}] {};

\fill (2,-2) circle (2pt) node[above] {2};
\fill (1.5,-1.5) circle (2pt) node[above] {2};
\fill (2.5,-1) circle (2pt) node[above] {2};
\fill (3,-1.7) circle (2pt) node[above] {2};

\path [draw=,snake it, ] (1,-3) -- (3,-3) node [label=below left:{$\PP(1,2)$}] {};
\draw (1.2,-3) circle (2pt) node[above] {1};
\draw (1.5,-3) circle (2pt) node[above] {1};
\draw (2,-3) circle (2pt) node[above] {1};
\draw (2.5,-3) circle (2pt) node[above] {1};

\path [draw=,snake it, ] (10.5,-3) -- (12.5,-3) node [label=below left:{$\PP(1,2)$}] {};
\fill (10.7,-3) circle (2pt) node[above] {2};
\fill (11,-3) circle (2pt) node[above] {2};
\fill (11.5,-3) circle (2pt) node[above] {2};
\fill (12,-3) circle (2pt) node[above] {2};
\end{tikzpicture}
\end{center}
The bottom right part, $\PP(1,2)$ with $4$ double points, is considered intrinsically in $\PP(1,2)$ and is $\AH_1(14)$ by \Cref{GK}, proving $(1)$ in \Cref{TERA}. To justify $(2)$ in \Cref{TERA}, we need the bottom left figure, regarded in $\PP(1,2,3)$, to be $\AH_2(11)$. First we show $4$ general double points are $\AH_2(11)$ then by \Cref{C13} their union with 4 simple points on $\PP(1,2)$ (empty circles) preserves the $\AH_2(11)$ property.

Since $s_{11}=16$, to show that any number of points is $\AH_2(11)$ it suffices to demonstrate $\LF\frac{1}{3} s_{11} \RF=5 $ and $\LC\frac{1}{3}s_{11} \RC=6 $ points are $\AH_2(11)$. We employ \Cref{TERA} again and focus on $5$ points which implies $4$ points are $\AH_2(11)$. Checking \Cref{eq: TERRA SET 1} and \Cref{eq: TERRA SET 2}, all three inequalities yield a value $q$. We choose $15-s_8=5\leq s'''_{11}=6$ which is to specialize $3$ double points to $\PP(1,2)$.
\vspace{.5cm}
\begin{center}
\begin{tikzpicture}[thick]
\fill (7,0) circle (2pt) node[above] {2};
\fill (7.5,.5) circle (2pt) node[above] {2};
\fill (7,-1) circle (2pt) node[above] {2};
\fill (6.5,-1) circle (2pt) node[above] {2};
\fill (7.5,-1) circle (2pt) node[above] {2};
\path [draw=,snake it, ] (6,-1) -- (8,-1) node [label=below left:{$\PP(1,2)$}] {};
\path [draw=,snake it]  (5.5,-1) -- (5.5,1)  node [midway, label= left :{$\PP(2,3)$}] {};
\path [draw=,snake it]  (8.5,-1) -- (8.5,1)  node [midway, label=right:{$\PP(1,3)$}] {};

\fill (2,-2) circle (2pt) node[above] {2};
\fill (2.5,-1.5) circle (2pt) node[above] {2};
\draw (2,-3) circle (2pt) node[above] {1};
\draw (2.5,-3) circle (2pt) node[above] {1};
\draw (1.5,-3) circle (2pt) node[above] {1};
\path [draw=,snake it, ] (1,-3) -- (3,-3) node [label=below left:{$\PP(1,2)$}] {};

\path [draw=,snake it, ] (10.5,-3) -- (12.5,-3) node [label=below left:{$\PP(1,2)$}] {};
\fill (12,-3) circle (2pt) node[above] {2};
\fill (11.5,-3) circle (2pt) node[above] {2};
\fill (11,-3) circle (2pt) node[above] {2};

\end{tikzpicture}
\end{center}
By \Cref{IND6}, $2$ general double points are $\AH_2(8)$ in $\PP(1,2,3)$. By \Cref{lem: Chandler} their union with $3$ simple points on $\PP(1,2)$ keeps the configuration $\AH_2(8)$. Hence any set of general double points in $\PP(1,2,3)$ is $\AH_2(14)$.
\eex 

In \Cref{TERANUM} we establish the numerical condition of \Cref{TERA}. 

\blm \label{TERANUM}
Let $r=\lfloor\frac{1}{3} s_d \rfloor$ or $r=\lceil \frac{1}{3} s_d \rceil$ and $d \geq 6$. Then the following equalities can not hold simultaneously.
\begin{equation} \label{eq: 3eq}
    \begin{split}
     &3r-s_{d-1}=s'_d= \text{positive odd integer} ,\\
     &3r-s_{d-2}=s''_d=\text{positive odd integer} ,\\
     &3r-s_{d-3}=s'''_d= \text{positive odd integer}        
    \end{split}
\end{equation}
Hence there always exists an even number, say $2q$, such that at least one of the following inequalities holds.
\begin{center}
\begin{minipage}{0.45\textwidth}
\begin{equation} \label{eq: TERRA SET 1}
\begin{split}
      3r-s_{d-1} &\leq 2q \leq s'_d\\
       3r-s_{d-2} &\leq 2q \leq s''_d\\
      3r-s_{d-3} &\leq 2q \leq s'''_d\\    
\end{split}
\end{equation}
\end{minipage}
\begin{minipage}{0.45\textwidth}
\begin{equation} \label{eq: TERRA SET 2}
\begin{split}
      s'_d &\leq 2q \leq 3r-s_{d-1}\\
      s''_d &\leq 2q \leq 3r-s_{d-2} \\
      s'''_d &\leq 2q \leq 3r-s_{d-3}\\
    \end{split}    
\end{equation}
\end{minipage}    
\end{center}
Furthermore, one can pick $q$ such that $1 \leq q \leq r$.
\elm

\bpf 
Assume for the sake of contradiction all equalities in \Cref{eq: 3eq} hold.
The identity $s'_d+s_{d-1}=s_d=3r$ follows from \Cref{equation: HF-NZD}. This implies $s_d$ is divisible by $3$ thus $r=\LF 1/3 s_d \RF = \LC 1/3 s_d \RC$.

Let $d=6k+i$ where $k\geq 1$ since $d\geq 6$. Plugging $d$ into the formulas from \Cref{FORMULA123} gives

\begin{enumerate}[label=\Roman*.]  
\item $s_d=\lfloor d^2/12+d/2+1 \rfloor=3k^2+(3+i)k+1+\lfloor i^2/12+i/2 \rfloor$

\item $s'''_d=\lfloor d/2 \rfloor +1=3k+\lfloor i/2 \rfloor +1=$ odd

\item $s''_d=\lfloor d/3 \rfloor +1=2k+\lfloor i/3 \rfloor+1=$ odd 

\item $s'_d =\begin{cases}
         k + \LF i/6  \RF   &d \equiv 1 \pmod 6   \\   
        k + \LF i/6 \RF +1 &\text{otherwise} 
    \end{cases}=$  odd
\end{enumerate}

\noindent Part III implies $i \neq 3,4,5$, otherwise the $s''_d$ is not odd. If $i=2$ then IV implies $k$ is even which implies part II, $3k+2$, is even. Hence $i\neq 2$. Since $3$ divides $s_d$, I implies $i\neq 0$. Finally suppose $i=1$. Part IV implies $k$ must be odd. Hence $3k+1$ is even which implies II can not hold. It follows that one of the 6 inequalities in \Cref{eq: TERRA SET 1} or \Cref{eq: TERRA SET 2} holds for any $d\geq 6$. This proves the existence of $q$.

It remains to show one can pick $q$ to satisfy $1 \leq q \leq r$. For $1\leq q$, it suffices to show either the lower bound in \Cref{eq: TERRA SET 2} or the upper bound in \Cref{eq: TERRA SET 1}, namely $s'_d$, $s''_d$ and $s'''_d$ is at least $2$ for $d\geq 6$. For all $i$ and by \Cref{equation: HF-NZD}
\begin{equation*}
    s''_i \leq s''_{i+1}, \; \; 
    s'''_i \leq s'''_{i+1},
\end{equation*}
and $s''_6=3$ and $s'''_6=4$. Therefore $1\leq q$ if either of the bottom two inequalities in \Cref{eq: TERRA SET 1} or \Cref{eq: TERRA SET 2} holds. If $d=6$ or $d\geq 8$ we also get $s'_d \geq 2$.
For $d=7$, observe that
    \begin{equation*}
        3r-s_{6}=\begin{cases}
           0 &\text{ if } r= \LF \frac{1}{3} s_7 \RF\\
           3 &\text{ if } r=\LC \frac{1}{3} s_7 \RC
        \end{cases} 
        \; ,\; s'_7=1.
    \end{equation*}
For $r= \LF \frac{1}{3} s_7 \RF$ the inequality $3r-s_{6} \leq s'_7$ holds true but we do not have a choice of $q$. But we can ignore this and use another inequality because
\begin{equation*}
    3r-s_{d-3}=3 \leq s'''_d=4    
\end{equation*}
also holds when $r=\LF \frac{1}{3} s_7 \RF $. Therefore $ 1\leq q$ for either value of $r$ and $d \geq 6$.

To show $q\leq r$, it is sufficient that either the lower bound in \Cref{eq: TERRA SET 1} or the upper bound in \Cref{eq: TERRA SET 2} is less than or equal to $2r$; that is $3r-s_{d-i} \leq 2r $ for $i \in \{1,2,3 \}$. In \Cref{eq: TERRA SET 1} this is sufficient because if $2r$ is less than or equal to the right hand side, we may choose $q$ to be the smallest number such that $2q$ is in the given interval. 
Note that
\begin{equation*}
    3r-s_{d-i} \leq 2r \Leftrightarrow r \leq s_{d-i}.
\end{equation*}
Moreover, considering $r=\LC \frac{1}{3}s_d \RC$ and $i=3$ is enough because
\begin{equation*}    
 \LF\frac{1}{3}s_d \RF \leq \LC\frac{1}{3} s_d \RC  \leq s_{d-3} \leq s_{d-2} \leq s_{d-1}
\end{equation*}
Replacing $z=x_i$ and $M=S$ in \Cref{equation: HF-NZD} for the equivalence $(*)$ yields
\begin{equation} \label{eq: Equivalence}
     \LC\frac{1}{3} s_d \RC  \leq s_{d-3} \Leftrightarrow \frac{1}{3} s_d < s_{d-3} \underset{*}{\Leftrightarrow} \frac{1}{3}s'''_d+\frac{1}{3}s_{d-3} < s_{d-3}  \Leftrightarrow s'''_d < 2\cdot s_{d-3}
\end{equation}
\Cref{NumFact 1} finishes the proof by showing \Cref{eq: Equivalence} for $d\geq 6$.
\epf 
\noindent \Cref{lem: Chandler} proves when \Cref{TERANUM} is satisfied a set $X$ of $r-q$ double points and $q$ simple points on $L$ is $\AH_2(d-i)$ for $i\in \{1,2,3 \}$. By \Cref{def: AH} this means
\begin{equation*}
    H_{d-i}(S/I_X)=\min \{ s_{d-i}, 3(r-q)+q\}.
\end{equation*}
By \Cref{TERANUM} we know only one of the six inequalities above may be satisfied for any given $d$ and specified values of $r$. Therefore the proof shows any one of the inequalities implies $(2)$ in \Cref{TERA}. The proof makes repeated use of \Cref{C13} and the induction hypothesis.

\blm \label{lem: Chandler}
Let $d\geq 6$. By \Cref{TERANUM} at least one of the inequalities
\begin{center}
\begin{minipage}{0.45\textwidth}
\begin{equation*}
\begin{split}
      3r-s_{d-1} &\leq 2q \leq s'_d\\
       3r-s_{d-2} &\leq 2q \leq s''_d\\
      3r-s_{d-3} &\leq 2q \leq s'''_d\\    
\end{split}
\end{equation*}
\end{minipage}
\begin{minipage}{0.45\textwidth}
\begin{equation}\label{eq: 6.8}
\begin{split}
      s'_d &\leq 2q \leq 3r-s_{d-1}\\
      s''_d &\leq 2q \leq 3r-s_{d-2} \\
      s'''_d &\leq 2q \leq 3r-s_{d-3}\\
    \end{split}    
\end{equation}
\end{minipage}    
\end{center}
is satisfied and $1 \leq  q \leq r$. 

Let $L$ denote the projective line with defining variable of degree $i$ where an inequality is satisfied. Assume for any $r'<r$ a general set  of $r'$  double points in $\PP(1,2,3)$ is both $AH_2(d-i)$ and $AH_2(d-2i)$. Then the union of a set of $r-q$ general double points in $\PP(1,2,3)$  and a set of $q$ general simple points of $L$ is $\AH_2(d-i)$.
\elm

\bpf 
Denote by $\overline{S}$ the coordinate ring of $L$ where $i\in \{ 1,2,3\}$ is the corresponding index chosen according to the inequality that holds in \cref{eq: 6.8}. 

Let $Y$ be a set of $r-q$ general double points in $\PP(1,2,3)\backslash L$ with  defining ideal $I_{q}$ and $J_{q} \subseteq S$ be the defining ideal of $q$ general simple points of $L$ perceived in $\PP(1,2,3)$. The ideal $I=I_{q}\cap J_{q}$ defines the union of $Y$ with $q$ simple points on $L$.

\noindent \textbf{Case 1:}
Suppose $3r-s_{d-i} \leq 2q \leq \overline{s}_d $.
 We want to show 
 \begin{equation*}
     H_{d-i}(S/I)=\min\{ s_{d-i}, 3(r-q)+q \}=3r-2q,
 \end{equation*} 
 where the last equality follows from the assumption $3r-s_{d-i} \leq 2q$. 
 Towards this end we apply \Cref{C13}. Condition (2) in this lemma combined with $Y$ being $AH_2(d-i)$ and $AH_2(d-2i)$ becomes
 \begin{equation*}
    \begin{split}
     H_{d-i}(S/I_q)+q \leq H_{d-2i}(S/I_q)+H_{d-i}(\overline{S})\\
    \Leftrightarrow  3r-2q \leq \min \{s_{d-2i},3(r-q) \}+\overline{s}_{d-i}.
    \end{split}
\end{equation*}
If $H_{d-2i}(S/I_q)=s_{d-2i}$, then the right hand side is $s_{d-2i}+\overline{s}_{d-i}=s_{d-i}$ and the inequality $3r-2q\leq s_{d-i}$ is the assumption of this case. 
On the other hand, if $H_{d-2i}(S/I_q)=3(r-q)$, we need to show
\begin{equation*}     
3r-2q \leq 3(r-q)+\overline{s}_{d-i} \Leftrightarrow q\leq \overline{s}_{d-i}
\end{equation*} 
Assumptions of this case imply $q \leq \overline{s}_d/2$. \Cref{NumFact 2} illustrates
\begin{equation}\label{eq: NumFact 2}
    \overline{s}_d/2 \leq \overline{s}_{d-i},
\end{equation} which establishes the desired inequality $q\leq \overline{s}_{d-i}$.

Since condition (2) of \Cref{C13} holds, by the equivalent condition (1) of this result the ideal $I=I_q\cap J_q$ satisfies
 \begin{equation*} \label{eq: CHAND}    
 H_{d-i}(S/I)=H_{d-i}(S/I_q)+q.
 \end{equation*} 
By hypothesis $Y$ is $\AH_2(d-i)$ and further $3(r-q) < 3r-2q \leq s_{d-i}$. Hence the displayed identity leads to the desired conclusion
\begin{equation*}
    H_{d-i}(S/I_q)=3(r-q)=\min\{s_{d-i},3(r-q)\}.
\end{equation*}

\noindent \textbf{Case 2:} Suppose $\overline{s}_d \leq 2q \leq 3r-s_{d-i}$. 
We wish to demonstrate
\begin{equation*}
    H_{d-i}(S/I)=\min \{ s_{d-i},3r-2q\}=s_{d-i}.
\end{equation*}
 Since $Y$ is $AH_2(d-i)$ we have
\begin{equation*}
    H_{d-i}(S/I_{q})=\min \{s_{d-i},3(r-q) \}.
\end{equation*}
If $H_{d-i}(S/I_{q})=s_{d-i}$, we obtain 
\begin{equation*}
    [I_{q}]_{d-i}=0 \Rightarrow [I]_{d-i}= [I_{q}]_{d-i} \cap [J_{q}]_{d-i}=0 \Rightarrow H_{d-i}(S/I)=s_{d-i}.
\end{equation*}
Suppose $H_{d-i}(S/I_{q})=3(r-q)<s_{d-i}$. Set $m=s_{d-i}-3(r-q)$ and note $m\leq q$ by assumption of this case. Let $J_m$ be the defining ideal of $m$ among the chosen simple points on $L$, $J_{q-m}$ the ideal defining the remaining simple points, and $F_m=I_{q} \cap J_m$. It is enough to show $(F_m)_{d-i}=0$. Then
\begin{equation*}
    [I]_{d-i}=[I_{q}]_{d-i}\cap [J_{m}]_{d-i} \cap [J_{q-m}]_{d-i} =[F_m]_{d-i} \cap [J_{q-m}]_{d-i}=0
\end{equation*}
implies $H_{d-i}(S/I)=s_{d-i}$.  
Equivalently we want to show
\begin{equation} \label{eq: CHAND2}
    H_{d-i}(S/F_m)=H_{d-i}(S/I_{q})+m=s_{d-i}.
\end{equation}
 By \Cref{C13}, \Cref{eq: CHAND2} is true if and only if
\begin{equation}\label{ineqq1}
    s_{d-i} \leq H_{d-2i}(S/I_{q})+\overline{s}_{d-i}=\min \{ s_{d-2i},3(r-q)\}+\overline{s}_{d-i},
\end{equation}
where the last equality follows from $Y$ being $AH_2(d-2i)$. It suffices to show $q \leq \overline{s}_{d-i}$. If true, we use the assumption $s_{d-i} \leq 3r-2q$ to conclude $s_{d-2i}=s_{d-i}-\overline{s}_{d-i}\leq 3(r-q)$. Then the right hand side of \Cref{ineqq1} becomes $s_{d-i}$ and the inequality claimed therein holds as an equality. Therefore we have established \Cref{eq: CHAND2}, which settles the claim as explained above.
To verify the remaining inequality $q \leq \overline{s}_{d-i}$, observe that
\begin{equation*}
    q\leq r \leq \LC \frac{1}{3} s_d \RC \text{ and } s_{d-3} \leq s_{d-2}\leq s_{d-1}.
\end{equation*}
The inequality $\LC \frac{1}{3} s_d \RC \leq s_{d-3}$ follows from \cref{eq: Equivalence} and \Cref{NumFact 1}. 
\epf

\section{Exceptions}\label{section: exceptions}
\subsection{An interpolation bound for $\PP(1,b,c)$} 
Let $k$ be an algebraically closed field of characteristic $0$. The main goal of this section to give an interpolation bound for weighted projective planes of the form $\PP(1,b,c)$ in \Cref{Theorem: 1bc bound}. That is to find $t \in \N $ such that for all $d\geq t$ we can guarantee that any set of $r$ general double points $X$ is $\AH_2(d)$, or not exceptional, i.e.
 \begin{equation*}
     H_d(S/I_X)=\min \{ s_d, 3r\}.
 \end{equation*} 
In general double point interpolation in $\PP(1,b,c)$ is drastically different than in $\PP(1,2,3)$ and $\PP_k^2$ as indicated by \cref{example: deficiency}.

\bdf \label{def: deficiency}
For a set of $r$ general double points $X\subseteq \PP(a,b,c)$ we define {\it deficiency of $X$ in degree $d$} as 
\begin{equation*} \label{eq; Deficiency}
   D(X,d):= \min\{ s_d,3r \} - H_d(S/I_X)
\end{equation*}
\edf 

By \Cref{remark: Weight of 1} when at least one weight is $1$ we have $D(X,d) \geq 0$.  Further using this terminology, \Cref{thm: AH} states that $D(X,d)=0$ in $\PP_k^2$ with exceptions when $d=2,\; r=2$ and $d=4,\; r=5$ and deficiency is exactly $1$ in each case. Computations with Macaulay2 \cite{M2} illustrate even in $\PP(1,b,c)$ this is not the case. 

\bex \label{example: deficiency}
We demonstrate some examples of exceptional cases in $\PP(1,b,c)$.
\begin{table}[h!]
    \setlength\belowcaptionskip{-10pt}
    \centering
    \begin{tabular}{|c|c|l|}
    \hline
          Weighted $\PP^2$ & $\#$ of general points & {\centering Deficiency}  \\ \hline   
        $\PP(1,5,9)$ & 3 & $D(X,d)=1$ for $20\leq d \leq 22$\\ \hline
           $\PP(1,5,26)$& 2 & \vtop{\hbox{\strut $D(X,d)=2$ for $d=25$}\hbox{\strut $D(X,d)=1$ for $d=20-24,26-30$}} \\ \hline
        $\PP(1,4,57)$& 4 & \vtop{\hbox{\strut $D(X,d)=1$ for $d=32-35, 65-68$}\hbox{\strut $D(X,d)=2$ for $d=36-39, 61-64$}\hbox{\strut $D(X,d)=3$ for $d=40-43, 57-60$}\hbox{\strut $D(X,d)=4$ for $d=44-56$}}\\
        \hline
          \end{tabular}
    \caption{Exceptional Cases in $\PP(1,b,c)$}
    \label{tab:my_label}
\end{table}
\eex
\Cref{example: deficiency} shows that even when the ideal of points in a weighted projective plane are defined by complete intersections, their behavior can be different than points in the usual $\PP^2_k$. Our Macaulay2 experimentation suggests by increasing the difference between the two largest weights one can arbitrarily increase $D(X,d)$ for a suitable number of general double points.
\begin{problem}
    If $d\in \N$ is allowed to vary, is $D(X,d)$ bounded in any weighted projective plane?
\end{problem}

Proof of \Cref{Theorem: 1bc bound} mimics Terracini's proof of \Cref{thm: AH} for $n=2$; \cite{Terra1}. It relies on two two prelimenary results, \Cref{equiv} and \Cref{Terracini 2}. We introduce the necessary terminology to state these results.

\bdf\label{def: Veronese Map}
Let $S=k[x_0,\ldots, x_n]$ be the coordinate ring of $\PP(a_0,\ldots,a_n)$, $N=s_d-1$, $m_0,\ldots, m_N$ enumerate the monomials in $S_d$ and $I$ be the kernel of the $k$-algebra homomorphism
\begin{equation*}
\begin{split}
\theta: k[y_0,\ldots,y_N] &\to k[x_0,\ldots, x_n] \\
y_i &\to m_i.    
\end{split}
\end{equation*}
The projective variety
\[
V^d= \Proj(k[y_0,\ldots,y_N]/I)\subset \PP^N_k
\] is called the {\em $d$-th Veronese variety} of $\PP(a_0,\ldots,a_n)$.
Geometrically $\theta$ induces a map of projective varieties 
\begin{equation*}
\psi:\PP(a_0,\ldots,a_n)\dashrightarrow V^d, \, \psi(p)=[m_0(p):\cdots: m_N(p)]
\end{equation*}
 defined on $U = \{ p \in \Proj(S) \; | \; \theta((y_0,\ldots, y_N)) \not \subseteq p \}$.
 \
\edf
Note that unlike the straight projective space, the Veronese map is only partially defined on $\PP(a_0,\ldots,a_n)$. Henceforth the notion of a set of general points $X$ includes $X\subseteq U$.

By $\sigma_r(V^d)$ we denote the $r$-th secant variety to the image of the Veronese map.  For a point $p \in V^{d}$, we let $T_pV^{d}$ denote the {\em tangent space of $V^{d}$ at p}. For a precise definition of each see \cref{def: secant variety} and \cref{def: tangent space}.

\bthm \label{equiv}
For a fixed $d\geq a_n$ and a set of $r$ general double points $X\subseteq \PP(1,a_1,\ldots,a_n)$, the following are equivalent:
\begin{enumerate}
    \item ${H_{d}(S/I_X)=\min\{s_d, r(n+1)\}}$
    \item $\dim_k(\sigma_r(V^d))=\min\{s_d-1, r(n+1)-1\}$.
\end{enumerate}
\ethm

\bpf 
See \Cref{App: Secant Variety}.
\epf 

\blm {\bf (Second Terracini lemma)}   \label{Terracini 2}
Let $U$ be as in \cref{def: Veronese Map}.
Let $X=\{2p_i \}_{i=1}^r \subseteq U \subseteq \PP(a_0,\ldots,a_n)$ be a set of double points with defining ideal $I_X$ and fix $d \geq a_n$. We identify the reduced points $p_i$ with their images on $V^{d}$ according to the Veronese map. 

Assume that $H_{d}(S/I_X)<\min\{s_d, r(n+1)\}$. Then there is a positive dimensional variety $C\subseteq V^{d}$
 through $p_1,\ldots p_r$
such that if $p\in C$ then
$T_pV^{d}\subseteq <T_{p_1}V^{d},\ldots , T_{p_r}V^{d}>$. In particular, by \Cref{LASKERR}, every hypersurface of degree $d$ which is singular at $p_i$ is also singular along $C$.
\elm 
\bpf 
See \cite[Lemma 2.3]{BO08}
\epf 
We note the proof of \Cref{Terracini 2} in \cite{BO08} is completely independent of weights and holds in the setting of weighted projective space.
\bthm \label{Theorem: 1bc bound}
A set of $r$ general double points $X\subseteq \PP(1,b,c)$ is $\AH_2(d)$,
\begin{equation}\label{eq: 1bc bound}
    H_{d}(S/I_X)=\min\{s_d,3r \},
\end{equation}
for all $d\geq 10c$. 
Additionally, \Cref{eq: 1bc bound} holds if $d \geq 6c$ and $\LF \frac{2c}{b} \RF \geq 5$.
\ethm
\bpf 
Let $X$ be a set of $r$ general double that does not satisfy \Cref{eq: 1bc bound} for some degree $d \geq c$. There exists $F \in (I_X)_d$ where $F \neq 0$. 
By \Cref{Terracini 2}, there is a positive dimensional variety $C \subseteq V^d$ through the image of the points in $V^d$. By abuse of notation we now identify $C$ with its preimage under the Veronese map. By the last part of \Cref{Terracini 2}, $F$ is singular along $C$. Hence $F$ contains a double component through $X$,  namely $F=C^2H$. Let $\deg(C)=l$.
Since one of the weights is $1$, \Cref{remark: Weight of 1}~(3) implies
\begin{equation} \label{eq: 1bc 1}
    2l \leq d \Rightarrow l \leq \left \lfloor \frac{d}{2} \right \rfloor \Rightarrow s_l \leq s_{\left \lfloor \frac{d}{2} \right \rfloor}.
\end{equation}
By \Cref{Terracini 2}, $C$ vanishes on the reduced set $\{p_i \}_{i=1}^r$. From \Cref{prop: simple point interpolation} and \Cref{rmk: specilization} we get
\begin{equation} \label{eq: 1bc 2}
   \left \lfloor \frac{1}{3}s_d \right \rfloor \leq r\leq s_l -1.
\end{equation}
Putting \Cref{eq: 1bc 1} and \Cref{eq: 1bc 2} together we acquire
\begin{equation} \label{eq: 1bc 3}
    \left \lfloor \frac{1}{3}s_d \right \rfloor \leq s_{\left \lfloor \frac{d}{2} \right \rfloor} -1.
\end{equation}
Hence for values of $d$ where \Cref{eq: 1bc 3} does not hold, any set $X$ satisfies \Cref{eq: 1bc bound}. In \Cref{lemma: 9c+1} we show that the inequalities involving $d$ given in the hypothesis force \Cref{eq: 1bc 3} to fail. This completes the proof.
\epf 

\blm \label{lemma: 9c+1}
If either of the following is true then $\left\lfloor \frac{1}{3} s_d\right \rfloor \geq s_{\LF \frac{d}{2} \RF}$ holds.
\begin{enumerate}
    \item$ d \geq 10c$
    \item $d \geq 6c$ and $\LF \frac{2c}{b}\RF \geq 5$
\end{enumerate}

\elm 

\bpf 
We employ an interpretation of the Hilbert function of a nonstandard graded polynomial ring as lattice point counting. 

Observe $s_d$ is the number of lattice points on and inside the triangle $T$ with vertices $(0,0), \left( 0, \frac{d}{c}\right), \left( \frac{d}{b},0\right)$. Recall $s_d$ is the number solutions to $r_0+br_1+cr_2=d$ where $r_i \in \N_0$. Which are in bijection with lattice points $(r_1,r_2)$ where $br_1+cr_2 \leq d$, $r_1\geq 0$ and $r_2\geq 0$. The latter is the number of lattice points on and inside of $T$ which we now divide into smaller triangles.
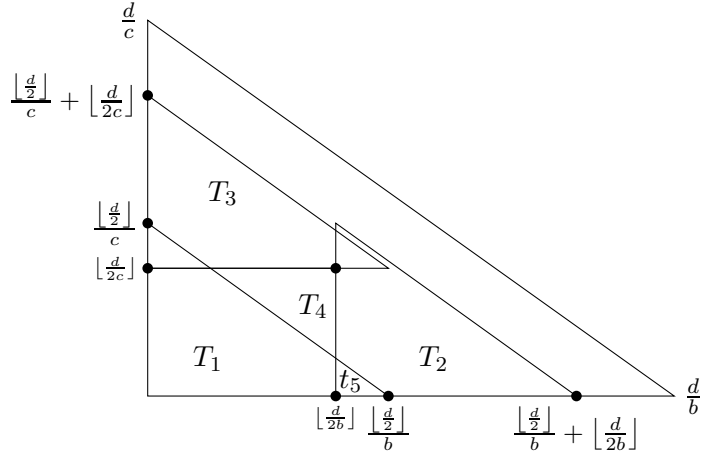
\begin{center} 
\begin{tikzpicture}
\draw (0,0) node[anchor=north]{}
  -- (7,0) node[anchor=west]{$\frac{d}{b}$}
  -- (0,5) node[anchor=south,left]{$\frac{d}{c}$}
  -- cycle;
\fill (3.2,0) circle (2pt) node[below] {\small $\frac{\LF \frac{d}{2} \RF}{b}$};
\fill (2.5,0) circle (2pt) node[below] {\tiny $\LF \frac{d}{2b} \RF$};
\fill (5.7,0) circle (2pt) node[below] {\small $\frac{\LF \frac{d}{2} \RF}{b}+\LF \frac{d}{2b} \RF$};

\fill (0,1.7) circle (2pt) node[left] {\tiny $\LF \frac{d}{2c} \RF$};
\fill (0,2.3) circle (2pt) node[left] {\small $\frac{\LF \frac{d}{2} \RF}{c}$};
\fill (0,4) circle (2pt) node[left] {\small $\frac{\LF \frac{d}{2} \RF}{c}+\LF \frac{d}{2c} \RF$};

\fill (2.5,1.7) circle (2pt) node[below] {};
(3.2,1.7) circle (2pt) node[below] {};
 (2.5,2.3) circle (2pt) node[] {};

\draw (3.2,0)--(0,2.3);
\draw (2.5,0) -- (2.5,2.3) -- (5.7,0);
\draw (0,4) -- (3.2,1.7)--(0,1.7);
\draw (2.5,1.7) -- (0,1.7);

\draw (.8,.5) node{$T_1$};
\draw (3.8,.5) node{$T_2$};
\draw (1,2.7) node{$T_3$};
\draw (2.2,1.2) node{$T_4$};
\draw (2.7,.2) node{$t_5$};
\end{tikzpicture}
\end{center}
Triangle $T_1$ has vertices at
$
    (0,0),\left (0,\frac{\LF \frac{d}{2}\RF}{c}\right ),\left (\frac{\LF \frac{d}{2}\RF}{b},0\right ) 
$.
Our goal is to impose conditions on $b$ and $c$ so that $T_1$, $T_2$ and $T_3$ have small intersections and $T_4$ contains enough lattice points to imply the desired inequality. By $\#T$ we refer to the number of lattice points on  and inside the triangle $T$.
Note that
\begin{itemize}
    \item $\#T_1=s_{\LF \frac{d}{2} \RF}$,
    \item $T_2$ and $T_3$ are translates of $T_1$ by integer vectors hence     $$\#T_2 = \#T_3= \#T_1=s_{\LF \frac{d}{2} \RF}.$$
\end{itemize}
We make the following observations
\begin{enumerate}
    \item $T_1 \cap T_2 =\left  \{ \left (\LF
    \frac{d}{2b}\RF,0 \right ) \right \}$.
    \item $T_2 \cap T_3 = \left\{\left (\LF \frac{d}{2b} \RF , \LF \frac{d}{2c} \RF \right) \right \}$
    \item $T_1 \cap T_3 \subseteq \left\{ \left(t, \LF \frac{d}{2c} \RF \right)\; | \: 0\leq t < \frac{c}{b} \right\}$
    \item $T_1\cap T_2 \cap T_3=\emptyset$.
\end{enumerate}
Item $(4)$ follows from $(2)$ once we check the point $\left(\LF \frac{d}{2b} \RF , \LF \frac{d}{2c} \RF \right)$ is not in $T_1$. The slope of the hypotenuse of $T_1$ is $\frac{-b}{c}$, so the point is in $T_1$ if and only if $\LF \frac{d}{2c}\RF \leq  \frac{b}{c}   < 1$. But this can't happen since $d\geq 2c$ implies $\LF\frac{d}{2c} \RF \geq 1$.

Using $(1)-(4)$ above and denoting by $\#  \accentset{\circ}{T}_4$ the number of lattice points in the interior of $T_4$, we obtain
\begin{align}\label{eq: 1bc 6}
    s_d &=\#T  \nonumber \\
    &\geq   \#T_1+\#T_2 +\#T_3- \#(T_1 \cap T_2) -\#(T_2 \cap T_3) -\#(T_1 \cap T_3)+ \#\accentset{\circ}{T}_4  \nonumber \\   
    & \geq  3s_{\LF \frac{d}{2} \RF}-3-\LF \frac{c}{b} \RF + \#\accentset{\circ}{T}_4. 
\end{align}
We now estimate $\#\accentset{\circ}{T}_4$. Let $(p,q)$ be the point of intersection between the hypotenuse of $T_1$ and the line $x=\LF \frac{d}{2b}\RF$. Using the figure above one can show 
\begin{equation} \label{eq: T4 interior}
   q < i < \LF \frac{d}{2c}\RF \text{ and } 
   \frac{\LF \frac{d}{2} \RF}{b}-\frac{ic}{b} < t < \LF \frac{d}{2b} \RF \iff  (t,i) \in \accentset{\circ}{T}_4   
\end{equation}
For a fixed $i$, we obtain $\LF\frac{ic}{b}  \RF-1 \geq i-1$ lattice points in $\#\accentset{\circ}{T}_4$.
Let $d \geq 10c$. From $\LF \frac{d}{2c} \RF \geq 5$ and \Cref{eq: T4 interior} we get $i\leq 4$.  Since  $t_5$ is similar to $T$ and $\frac{\LF \frac{d}{2} \RF}{b} - \LF \frac{d}{2b} \RF \leq 1$ we get $q< 1$, hence $1\leq i$. Plugging $1\leq i \leq 4$ in \Cref{eq: T4 interior} one gets
\begin{equation*}
    \#\accentset{\circ}{T}_4 \geq \LF \frac{c}{b} \RF +5.
\end{equation*}
Plugging this into \Cref{eq: 1bc 6} yields the following and proves $(1)$.
\begin{equation*}
    s_d \geq 3s_{\LF \frac{d}{2} \RF} - \LF \frac{c}{b} \RF -3 + \LF \frac{c}{b} \RF +5= 3s_{\LF \frac{d}{2} \RF}+2.
\end{equation*}
Suppose $d \geq 6c$ and $\LF \frac{2c}{b} \RF \geq 5$. By $\LF \frac{d}{2c} \RF \geq 3 $ and \Cref{eq: T4 interior} we get 
\begin{equation*}
    \#\accentset{\circ}{T}_4 \geq \LF \frac{c}{b} \RF -1 + \LF \frac{2c}{b} \RF -1.
\end{equation*}
By $\LF \frac{2c}{b} \RF \geq 5 $ and \Cref{eq: 1bc 6}  we arrive at the desired conclusion
\begin{equation*}
    s_d \geq 3s_{\LF \frac{d}{2} \RF} - \LF \frac{c}{b} \RF -3 + \LF \frac{c}{b}\RF -1 + \LF \frac{2c}{b} \RF-1 \geq 3 s_{\LF \frac{d}{2} \RF}.  \qedhere
\end{equation*}
\epf 

\subsection{Exceptions in $\PP(1,b,c)$.}
So far we have shown exceptional sets of points in $\PP(1,b,c)$ only happen when $d < 10c$. As a porism of \Cref{Theorem: 1bc bound} we get \cref{Cor: Exceptions 3|s_d} which fully classifies exceptional cases of $\PP(1,b,c)$ in degrees $d$ where $s_d$ is divisible by $3$.
\Cref{prop: suff exception} gives a sufficient condition for a set of points to fail \Cref{eq: 1bc bound}.
\bpr \label{prop: suff exception}
Fix $d \in \N$. Let $X \subseteq \PP(1,a_1,\ldots,a_n)$ be a set of $r$ general simple points. Let $J_X$ and $I_X$ be the defining ideal of $X$ and $2X$ respectively; with $2X$ denoting the set of double points supported on $X$. If $r < s_{\LF \frac{d}{2} \RF}$ and $(n+1)r \geq s_d$ then $X$ is not $\AH_n(d)$.
\epr 
\bpf 
Let $e=\LF \frac{d}{2} \RF$. By \Cref{prop: simple point interpolation} and the assumption $r <s_e$, there exists $F \in (J_X)_{e}$ where $F \neq 0$.
The proof follows from the following simple observation.
\begin{equation*}
    F^2 \in (J^2_X)_{2e} \subseteq (I_X)_{2e}
\end{equation*}
However the expected dimension of $(I_X)_{2e}=0$ because
\begin{equation*}
    (n+1)\cdot r \geq s_d \underset{*}{\geq} s_{2e},
\end{equation*}
where $(*)$ follows since $a_0=1$ (\cref{remark: Weight of 1}~(3)). Thus $X$ is not AH$_n(d)$.
\epf 
\bcor \label{Cor: Exceptions 3|s_d}
Let $d \in \N$ be such that $s_d=3 r$. Then a set $X \subseteq \PP(1,b,c)$ of $r$ general double points is not $\AH_2(d)$ if and only if $r < s_{\LF \frac{d}{2} \RF}$.
\ecor
\bpf 
Let $X$ be as defined in the hypothesis. Using \Cref{prop: suff exception} for $(\Rightarrow)$ and \Cref{eq: 1bc 3} for $(\Leftarrow)$, which is a consequence of not being $\AH_2(d)$ as shown in the proof of \Cref{Theorem: 1bc bound}, we get
\begin{equation*}
\begin{split}    
    &r= \frac{1}{3} s_d  \leq s_{\LF \frac{d}{2} \RF}-1 < s_{\LF \frac{d}{2} \RF}\\
    &\iff \text{ X does does not impose independent conditions in degree $d$.} \qedhere
\end{split}
\end{equation*}
\epf 

Using \Cref{Cor: Exceptions 3|s_d} one can prove $\PP(1,2,3)$ is the only space of the form $\PP(1,b,c)$ with no exceptional sets of general double points. 
\bthm \label{thm: 123 ONLY}
Let $X \subseteq \PP(1,b,c)$ be a finite set of general double points. Then $X$ is $\AH_2(d)$ for all $d \in \N$ if and only if $b=2$ and $c=3$.
\ethm
\bpf 
$(\Rightarrow)$ The proof is broken into the cases $b=1$, $b=2$ and $b\geq 3$ where $r$ denotes the number of points in $X$. 
If $b=1$, $c=2$ and $r=3$ then $X$ is exceptional because $s_2=4$ and by \cref{prop: simple point interpolation} there exists a quadratic form $F$ through the $3$ points resulting in $F^2\in [I_X]_4$. But the fact that $s_4=9$ yields
\begin{equation*}
    H_{S/I_X}(4)\leq 8<9=\min\{9,9 \}. 
\end{equation*}
If $c\geq 3$ by \Cref{NECK} one general double point is exceptional.

For $b=2$ and $r=1$, \Cref{1POINT} implies $c<4$ and further $2<c$ by \Cref{rmk: well-formed}, therefore $c=3$ and we are done.
Let $b\geq 3$. We contradict the hypothesis by proving it can not hold for $r=1,2,3$ points and all $d$ simultaneously. When $r=2$, \Cref{2POINTS}, which uses the assumption $b\neq 2$, implies $\frac{3}{2}b<c<2b$ or $2b<c<3b$. By \Cref{1POINT} the latter would contradict $1$ point being $\AH_2(d)$ for all $d$ therefore $\frac{3}{2}b<c<2b$. This criteria on $c$ results in $s_{4b}=9$  and $s_{2b}=4$ thus \Cref{Cor: Exceptions 3|s_d} implies $3$ points are not $\AH_2(d)$. \Cref{AH123} yields the backwards direction.
\epf 
\section{Open problems} \label{section: Open Problems}
We end this paper with a list of open problems, some of which have been stated through out the paper. 
\begin{enumerate}
\item With possible restrictions on the weights, find explicit generators for the defining ideal of a general point in $\PP(a_0,\ldots,a_n)$ for $n\geq 3$.
\item Using the notation introduced in \Cref{section: exceptions},  is $D(X,d)$ bounded {\em uniformly} (independently of $d$) in any weighted projective plane?
\item Let $\PP(1,a_1,\ldots,a_n)$ and $g=\lcm(a_1,\ldots,a_n)$. One can show if $g\mid d$ then the Veronese map in \Cref{def: Veronese Map} becomes an embedding. Compute the minimal free resolution of the ideal defining the Veronese variety $V^d$ for such $d$.
\item The well-known SHGH conjecture, see \cite{IdealsofPowers}, states the exceptional degrees for a set $X \subseteq \PP^2$ where $X$ contains points with mixed multiplicity. State the analogue of this conjecture for $\PP(1,b,c)$.
\item We suspect that m\'{e}thode d'Horace is applicable in $\PP(1,a_1,\ldots,a_n)$ for $d\geq \lcm(a_1,\ldots,a_n)$. But we could not adapt the details of the modern proof due to Chandler, in particular the use curvilinear subschemes, in the setting of weighted projective space. We leave it as a conjecture that m\'{e}thode d'Horace, or a simple modification of it, applies in weighted projective space.
\end{enumerate}
\appendix

\section{Proofs related to \Cref{123}} \label{Appendix: 123}


This section contains proofs for several foundational results used in \Cref{123} adapted from \cite{HM21} and \cite{Chandler01} as well as proofs for some numerical inequalities we employ in that section. We use \Cref{not: s_d}. Our proof of \Cref{TERA} mimics the one presented in \cite{HM21}; with expansion on two details that were not clear to us. The proof of \Cref{lem: Chandler} gives an accessible proof of \cite[Lemma 3]{Chandler01} when all the weights are $1$.

Let $\mathfrak{m}$ be the homogeneous maximal ideal of a graded ring and let $I$ be any ideal. The ideal $(I:\mathfrak{m}^{\infty})$ is called {\em the saturation of $I$}, denoted $I^{\sat}$.
\blm [Generalized Castelnuovo's Inequality] \label{GCI}
Let $S$ be a not necessarily standard graded polynomial ring, $I$ a homogeneous ideal and $l$ a form of degree $k$ in $S$. Set $\tilde I =I: l$, $\overline{S}=S/(l)$ and $\overline{I}=(I+(l))/(l)$, the ideal of $I$ in $\overline{S}$. Then 
\[
H_{d}(S/I)\geq H_{d-k}(S/\tilde I)+H_{d}(\overline{S}/(\overline{I})^{\sat}),
\]
with equality if and only if $\overline{I}$ is saturated in $\overline{R}$.
\elm 

\bpf
See \cite[Lemma 2.2]{HM21}.
\epf 
We note the assumption on the ambient space in \Cref{TERA} allows us to conclude the defining ideal of a simple point is a complete intersection (\Cref{POINTS ARE CI}, \Cref{defidplane}) which is implicitly used through out the proof. 
For a set of reduced points $X$ with defining ideal $J_X$, let $2X$ denote the set of double points supported on $X$ where $I_X=J\symb_X$; see \Cref{def: symb Power Def}.

\begin{theoremTERA}[Generalized Terracini's Inductive argument]
\ \\
Consider $\PP=\PP(1,a_1,\ldots,a_n)$, with the coordinate ring $S=k[x_0,\ldots,x_n]$ where $\deg(x_i)=a_i$. Assume $1\leq q \leq r$ and $d\in \N$ satisfy one of the following inequalities for at least one $i$ in the range $0\leq i\leq n$,
\begin{equation} \label{eq: TERRACINI CONDITION2} 
    (n+1)r-s_{d-a_i} \leq nq \leq \overline{s}_d \qquad \text{ or } \qquad  \overline{s}_d \leq nq \leq (n+1)r-s_{d-a_i},
\end{equation}
where $\overline{S}=S/(x_i)$ is the coordinate ring of the hyperplane $L=V(x_i)\cong \PP(a_0,\ldots,\hat{a_i},\ldots,a_n)$ and $\overline{s}_d=\dim_k(\overline{S}_d)$.
If
\begin{enumerate} 
    \item a set of $q$ general double points in $L$ is $\AH_{n-1}(d)$ and
    \item the union of a set of $r-q$ general double points in $\PP$ and a set of $q$ general simple points in $L$ is $\AH_n(d-a_i)$
\end{enumerate}
then a set of $r$ general double points in $\PP$ is $\AH_n(d)$.
\end{theoremTERA}

\bpf
If $n\geq 3$  all general points in $\PP$ are considered in $U=\displaystyle \bigcup_{\{ j\; \mid\;  a_j=1\} } U_j$.
Let $Y_1$ be a set of $q$ general simple points in $L$, with $\overline{J_{Y_1}} \subseteq \overline{S}$ as its defining ideal. Note $Y_1 \subseteq \PP$ has the defining ideal $J_{Y_1}=(\overline{J_{Y_1}},x_i)S$. Let $Y_2$ be a set of $r-q$ general simple points in $\PP\bs L$ with defining ideal $J_{Y_2}\subseteq S$.  Let $I=J_{Y_1}^{(2)}\cap J_{Y_2}^{(2)}$ be the defining ideal of $2Y=2Y_1 \cup 2Y_2$ and $\overline{I}=\cfrac{I+x_i}{x_i}$. We show $2Y$ has the $\AH_2(d)$ property. 

For any $d$, the definition of Hilbert function and \Cref{def: EvaluationmMap} imply $H_d(S/I) \leq s_d$ and $H_d(S/I)\leq r(n+1)$ hence it suffices to show $H_d(S/I)\geq \min \{ s_d,r(n+1)\} $. Since $Y_1 \subseteq L$ and $Y_2 \subseteq \PP\bs L$ we get
\begin{equation*}\tilde I := (I:x_i)= (J^{(2)}_{Y_1} \cap J^{(2)}_{Y_2}:x_i)= (J^{(2)}_{Y_1}:x_i) \cap (J^{(2)}_{Y_2}:x_i)=J_{Y_1} \cap J^{(2)}_{Y_2}.    
\end{equation*}
To justify the final equality we will show $(J^{(2)}_{Y_1}:x_i) \subseteq J_{Y_1}$. Observe \begin{equation*}
s \in (J^{(2)}_{Y_1}:x_i)=(\displaystyle \bigcap_{j=1}^q J^{(2)}_{p_j}:x_i)=\bigcap_{j=k}^q (J^{2}_{p_j}:x_i) \Rightarrow s\cdot x_i \in     \bigcap_{i=k}^q J^{2}_{p_j}.
\end{equation*} We want to show $s \in J_{p_j}$ for all $j$.

Fix a $j$ and note $s\cdot x_i=0$ in $\gr_{J_{p_j}}S$, where $s \in S/J_{p_j}$ and $x_i \in J_{p_j}/J^2_{p_j}$, but clearly $x_i \not \in J^2_{p_j}$. Using the generators of $J_{p_j}$ from \Cref{POINTS ARE CI} or \Cref{defidplane}, we can show $S_{J_{p_j}}$ is a regular ring, hence $\gr_{J_{P_j}}S$ is a domain. Therefore $s \in J_{p_j}$ for all $j$.
Note $\tilde I$ defines a set of $q$ general simple points in $L$ and $r-q$ general double points in $\PP$.\\
Claim: $(\overline{I})^{\sat}={\overline{J_{Y_1}}}^{(2)}$. Observe that $\overline{I} \subseteq \overline{J_{Y_1}}$. Since $\overline{S}/\overline{J_{Y_1}}$ is a one-dimensional ring, $(\overline{I})^{(\sat)}$ is the intersection of minimal components of $\overline{I}$. These are minimal primes containing $I=J_{Y_1}^{(2)}\cap J_{Y_2}^{(2)}$ and $x_i$. Since no minimal prime of $J_{Y_2}$ contains $x_i$, we get $\Min\{ (I,x_i) \}=\Min \{ J_{Y_1}\}$. For any $\fp \in \Min\{ (I,x_i) \}$ we get 
\begin{equation*}
\begin{split}
(I,x_i)^{\sat}_\fp=(I,x_i)_\fp= (J_{Y_1}^{(2)}\cap J_{Y_2}^{(2)},x_i)_\fp &= (J_{Y_1}^{(2)},x_i)_\fp\cap (J_{Y_2}^{(2)},x_i)_\fp\\
&=(J_{Y_1}^{(2)},x_i)_\fp=(J_{Y_1}^{(2)},x_i)^{\sat}_\fp.       
\end{split}
\end{equation*}
This proves the claim. By \Cref{GCI} we get
\begin{equation}\label{TERRAG}
H_{d}(S/I) \geq H_{d-a_i}(S/\tilde I) + H_{d} (\overline{S}/\overline{I}^{\sat}) = \min\{s_{d-a_i},(n+1)(r-q)+q \} + \min \{\overline{s}_d,nq \}     
\end{equation}
where the last equality follows from assumptions $(1)$ and $(2)$.

If $(n+1)r-s_{d-a_i} \leq nq \leq \overline{s}_d$ holds, the right hand side of \Cref{TERRAG} becomes $(n+1)r$. If $ \overline{s}_d \leq nq \leq (n+1)r-s_{d-a_i}$ holds, the right hand side of \Cref{TERRAG} is $s_{d-a_i}+\overline{s}_d=s_d$ by \cref{rmk: SES-NZD}. Note if $H_{d}(S/I)$ is equal to $(n+1)r$ or $s_d$ then that value is the minimum of the two.
\epf

\begin{LemmaChandler}[Chandler's Lemma] 
Let $I$ be a saturated homogeneous ideal in $S=k[x_0,\ldots,x_n]$ with $\deg(x_i)=a_i$, the homogeneous coordinate ring of $\PP(a_0,\ldots,a_n)$. Assume $l$ is regular on $S/I$. Let $\overline{S}=S/(l)$ and $\deg(l)=i$. Fix $d \in Z_+$. Then following are equivalent:
\begin{enumerate}
    \item A general set $Y_0$ of $q$ reduced points in $V(l)$ satisfies
    \begin{equation}\label{eq: q points}
        H_{d}(S/(I \cap J_{Y_0}))=H_{d}(S/I)+q. 
    \end{equation}
    \item $H_{d}(S/I)+q \leq H_{d-i}(S/I)+H_{d}(\overline{S})$.
\end{enumerate}
\end{LemmaChandler}
\bpf 
$(1)\Rightarrow (2)$ Since $1\cdot l \in J_{Y_0}$ and $l$ is regular on $S/I$ we get 
$$(I \cap J_{Y_0}:(l))=(I:(l)) \cap (J_{Y_0}:(l))=(I:(l))\cap S =(I:(l))=I $$
Using this equality and the restriction sequence 
$$ 0\to S/(I\cap J_{Y_0}:l)(-i) \xrightarrow{\cdot l} S/(I\cap J_{Y_0}) \twoheadrightarrow S/(I\cap J_{Y_0},l) \to 0 $$
we get
$$H_{d}(S/I\cap J_{Y_0})=H_{d-i}(S/I)+H_{d}(S/(I\cap J_{Y_0},l)) \leq H_{d-i}(S/I)+H_{d}(\overline{S}) .$$

\noindent $(2)\Rightarrow (1)$ Suppose there exists a positive integer $q$ satisfying the inequality in (2). We proceed by induction on $j$ where 
$$ 0 \leq j \leq H_{d-i}(S/(I:l))+H_{d}(\overline{S})-H_{d}(S/I)$$
to construct a set $Y(j)$ of $j$ reduced points in $V(l)$ that satisfies  
\begin{equation}\label{eq: j points}
H_{d}(S/(I \cap J_{Y(j)}))=H_{d}(S/I)+j. 
\end{equation}

If $j=0$, let $Y(0) =\emptyset$. For $0\leq j< q$ suppose
there is a set $Y(j) \subseteq V(l)$ that satisfies \Cref{eq: j points}. Since $j+1\leq q$ assumption (2) gives
 \[
 j+1 \leq q\leq  H_{d-i}(S/(I:l))+H_{d}(\overline{S})-H_{d}(S/I)
 \]
from which it follows 
 $$H_{d}(S/(I \cap J_{Y(j)})) - H_{d-i}(S/(I:l)) \leq H_{d}(\overline{S})-1. $$
Recall that $(I \cap J_{Y(j)}:l)=(I:l)$ hence
\begin{equation}\label{idk}
H_{d}(S/(I \cap J_{Y(j)})) - H_{d-i}(S/(I \cap J_{Y(j)}:l)) \leq H_{d}(\overline{S})-1
\end{equation}
Using \Cref{idk} and the restriction sequence
$$ 0 \to S/(I \cap J_{Y(j)}:(l))(-i) \xrightarrow{\cdot l} S/(I\cap J_{Y(j)}) \twoheadrightarrow S/(I \cap J_{Y(j)}+l) \to 0$$
We get
$$1 \leq H_{d}(\overline{S})-H_{d}(S/(I \cap J_{Y(j)}+l))=H_{d}((I \cap J_{Y(j)}+l)/(l)) $$
Therefore there exists $F \in [I \cap J_{Y(j)}+l]_d \setminus [(l)]_d$. Let $p \in V(l)\setminus V(F)$. Such $p$ exists because $F\not\in (l)$. Further note
\begin{equation*}
    F \in ((I \cap J_{Y(j)}+l)/l) \Rightarrow V(F)\supseteq V(I \cap J_{Y(j)}) \cap V(l) \supseteq Y(j).
\end{equation*}
Since $p \not \in V(F)$ we get $Y(j+1)=Y(j) \cup \{p\}$ has $j+1$ distinct points. 
We obtain 
$$H_{d}(S/(I \cap J_{Y(j+1)}))
\geq  H_{d}(S/ (I \cap J_{Y(j)}))+1=H_{d}(S/I)+j+1,$$
where the first inequality follows since $F \in [I \cap J_{Y(j)}]_d\setminus [I \cap J_{Y(j+1)}]_d$ and the last equality follows from the induction hypothesis. Finally, the exact sequence below yields the other inequality.
$$ 0 \to \cfrac{S}{I \cap J_{Y(j+1)}} \to S/I \oplus S/J_{Y(j+1)}\to S/I+J_{Y(j+1)}\to 0$$
\begin{equation} \label{eq: j points upper bound}
    H_{d}(S/(I \cap J_{Y(j+1)})) \leq H_{d}(S/I) + H_{d}(S/J_{Y(j+1)}) \leq H_{d}(S/I)+j+1.
\end{equation}

We have shown the set $Y(q)$ satisfies \Cref{eq: j points}. By semicontinuity of Hilbert functions a general set $Y_0$ of $q$ points then satisfies $H_{d}(S/(I \cap J_{Y_0})) \geq H_{d}(S/I)+q$. As shown in \Cref{eq: j points upper bound}, the opposite inequality is always satisfied, so \Cref{eq: q points} follows.
\epf 

In the remainder of the section we prove different numerical criteria that are used in \Cref{123}. \Cref{not: s'} therein is in effect.

\bpr \label{NumFact 1}
Let $d\in \N$ satisfy $d\geq 6$. Then $ s'''_d<2\cdot s_{d-3} $.
\epr 

\bpf 
The table below shows the inequality for $6\leq d \leq 8$.
\begin{center}
\begin{tabular}{ |c|c|c|c| }
 \hline
 & $d=6$ & $d=7$ & $d=8$ \\
 \hline
 $s'''_d$ & $4$ & $4$ & $5$  \\
 \hline
 $2\cdot s_{d-3}$ & $6$ & $8$& $10$ \\
 \hline 
 \end{tabular}
 \end{center}
Next we show $s'''_{d+3} <2\cdot s_d$ for all $d\geq 6$. Let $d=6k+i$ where $1\leq k$ and $ 0\leq i \leq 5$. By \Cref{FORMULA123} we get 
\begin{equation*}s'''_{d+3}=3k+\LF\frac{i+3}{2} \RF + 1 \leq 3k+5 \text{ since } i \leq 5.    
\end{equation*}
Furthermore, by \Cref{FORMULA123} (1) we have $s_{6k}=3k^2+3k+1\leq s_d$. Since $1 \leq k$ the strict inequality holds and establishes
\begin{equation*}
    s'''_d \leq 3k+5 < 2\cdot s_{6k}=2\cdot(3k^2+3k+1) \leq 2\cdot s_d \qedhere 
\end{equation*} 
\epf 

\begin{lemma} \label{NumFact 2}
Let $d\in N$ be such that $d\geq 6$. We state and prove all cases of the inequality \Cref{eq: NumFact 2} used in the proof of \cref{lem: Chandler}.
\begin{enumerate}
    \item $\frac{s'_d}{2} \leq s'_{d-1}  $
    \item $\frac{s''_d}{2} \leq s''_{d-2}  $
    \item $\frac{s'''_d}{2} \leq s'''_{d-3}$
\end{enumerate}
\end{lemma}

\bpf 
$(1)$ Let $d=6p+i$ where $p\geq 1$. Using \Cref{FORMULA123} we have
\begin{equation*}    
s'_d=
 \begin{cases}
          p   &i=1   \\   
        p +1 &\text{otherwise} 
\end{cases}
\qquad \text{ and } \qquad 
s'_{d-1}=
 \begin{cases}
          p   &i=0,2   \\   
        p+1 &\text{otherwise.} 
\end{cases}
\end{equation*}
Hence it follows
\begin{equation*}
    \frac{s'_d}{2} \leq \frac{p+1}{2} \leq p \leq  s'_{d-1}.  
\end{equation*}

\noindent $(2)$ Let $d=3m+j$ where $m \geq 2$. 
Using \Cref{FORMULA123} we have
\begin{equation*} 
s''_d=\left \lfloor \frac{d}{3}\right \rfloor+1
\qquad \text{ and } \qquad
s''_{d-2}=
 \begin{cases}
          m   &j=0,1   \\   
        m+1 &j=2.
\end{cases}
\end{equation*}
Then
\begin{equation*}
\frac{s''_d}{2}= \frac{m+1}{2} \leq m  \leq s''_{d-2}.
\end{equation*}

\noindent $(3)$ Let $d=2s+t$ where $s\geq 3$. Recall from \Cref{FORMULA123}  that
\begin{equation*}  
s'''_d=\left \lfloor \frac{d}{2} \right \rfloor+1
\qquad \text{ and } \qquad
s''_{d-3}=
 \begin{cases}
          s-1   &t=0   \\   
        s &t=1.
\end{cases}
\end{equation*}
 It follows that
\begin{equation*}
    \frac{s'''_d}{2}=\cfrac{s +1 }{2} \leq s-1\leq s'''_{d-3}.\qedhere
\end{equation*}.
\epf 

\section{Secant varieties to the Veronese variety} \label{App: Secant Variety}
\begin{definition}\label{def: Veronese Variety}
Let $S$ be a polynomial ring. Define the {\em $d$-th Veronese variety} as $V^d=\Proj(k[S_d])$.
\end{definition}

\begin{remark}\label{rem: veronese embedding}
The Veronese variety can be embedded in (straight) projective space as follows:
 Let $S=k[x_0,\ldots, x_n]$ and $N=s_d-1$. Let $m_0,\ldots, m_N$ enumerate the monomials in $S_d$. Consider the $k$-algebra homomorphism
\begin{equation}\label{eq: veronese map}
\begin{split}
\theta: k[y_0,\ldots,y_N] &\to k[x_0,\ldots, x_n] \\
y_i &\to m_i.    
\end{split}
\end{equation}
Set $R=k[y_0,\ldots,y_N]$ to be the domain of $\theta$ and $I$ to be the kernel of $\theta$. Then \Cref{eq: veronese map}
gives rise to an isomorphism $R/I\cong k[S_d]$ which in turn induces an isomorphism of projective varieties $V^d\cong \Proj(R/I)\subset \Proj(R)= \PP^N_k$. 

Geometrically $\theta$ induces a map of projective varieties 
\begin{equation}\label{eq: psi}
\psi:\PP\dashrightarrow V^d, \psi(p)=[m_0(p):\cdots: m_N(p)]
\end{equation}
 defined on $U = \{ p \in \Proj(S) \; | \; \theta(R_+) \not \subseteq p \}$.
\end{remark}

The main goal of this section is to establish the following proposition.
\begin{TheoremSecantVariety}
Let $\PP=\PP(1,a_1,\ldots,a_n)$. For a fixed $d\geq a_n$ the following are equivalent:
\begin{enumerate}
    \item if $X$ is a set of $r$ general double points, ${H_{d}(S/I_X)=\min\{s_d, r(n+1)\}}$
    \item $\dim_k(\sigma_r(V^d))=\min\{s_d-1, r(n+1)-1\}$.
\end{enumerate}
\end{TheoremSecantVariety}

To show this we make use of tangent spaces which we now define.

\begin{definition} \label{def: tangent space}
Let $p$ be a point of an affine variety $U$ with defining ideal $\fm_p$. The {\em Zariski tangent space} to $U$ at $p$ is the vector space 
\[
T_p(U)=\Hom_k(\fm_p/\fm_p^2,k).
\] 

If $p$ is a point of a projective variety $X$, the projective tangent space to $X$ at $p$, denoted $\T_p(X)$ is the projectivization of the $T_{\widetilde{p}} \widetilde X$, where $\widetilde{X}$ is affine cone over $X$ and $\widetilde{p}$ is any point of the affine cone that corresponds to $p\in X$.
\end{definition}

\begin{lemma}\label{lem: differential map}
Consider a map  $g:\AA_k^n \to \AA_k^N$ given by $g(p)=(g_1(p),\ldots, g_N(p))$, where $g_i\in k[x_1,\ldots, x_n]$. For any $p\in \AA^n$ the induced map
\[
\d g:T_p\AA^n\to T_{g(p)} \AA^N \text{ is represented by the Jacobian matrix } \begin{bmatrix} \frac{\partial g_i}{\partial x_j}\big |_p \end{bmatrix}_{\substack{1\leq i\leq N\\ 1\leq j\leq n}}.
\]
\end{lemma}

We now analyze the projective tangent space to the Veronese variety $V^d$.

\bpr \label{LASKERR}
Let $\PP=\PP(1,a_1,\ldots,a_n)$, let $d\geq a_n$ be an integer and  denote $N=s_d-1$. 
Using the notation introduced in \Cref{eq: veronese map} and \Cref{eq: psi}, consider a point $p\in U_0=\PP\setminus V(x_0)$ and set $U'_0=V^d\cap(\PP_k^N\setminus V(y_0))$. Then 
\begin{enumerate}
\item The inclusion $k[S_d]\hookrightarrow S$ induces an isomorphism between $U_0$ and $U_0'$, both of which are isomorphic to $\AA_k^n$. In particular, the points $p$ and $\psi(p)$ are smooth on $\PP$ and $V^d$, respectively.
\item The point $\psi(p)$ is defined and the projective tangent space $\T_{\psi(p)}(V^d)$ is given by
\begin{equation}\label{eq: tangent space}
\T_{\psi(p)}(V^d) \cong \PP \left \langle \left [ \frac{\partial m_0}{ \partial x_i} (p),\cdots , \frac{\partial  m_N}{ \partial x_i}(p)\right ] \; \Big | \; 0\leq i \leq n\right \rangle.
\end{equation}
\end{enumerate}
\epr

\begin{proof}
(1) Let $S=k[x_0,\ldots, x_n]$ with $\deg(x_i)=a_i$. Recall that $\PP=\Spec(S)$. Moreover the coordinate ring of $U_0$ is $S[x_0^{-1}]_0\cong k[x_1,\ldots, x_n]$ since $\deg(x_0)=1$. See the description of the coordinate rings of $U_i$ in \Cref{eq: affine charts}  for a justification of this isomorphism. Since the coordinate ring of $U_0$ is regular, all points of $U_0$ are smooth.

Now consider $V^d=\Spec(k[S_d])$. 
The following are monomials in $S_d$ 
\[
m_0=x_0^d, m_1=x_0^{d-a_1}x_1, \ldots,  m_n=x_0^{d-a_n}x_n, 
\]
and every other monomial $m$ in $S_d$ can be written in the form
\[
m=\prod_{i=0}^n x_i^{e_i}=\frac{x_0^{e_0}\prod_{i=1}^n m_i^{e_i}}{x_0^{\sum_{i=1}^n e_i(d-a_i)}}=\frac{x_0^{\sum_{i=0}^n e_ia_i}\prod_{i=1}^n m_i^{e_i}}{x_0^{\sum_{i=1}^n e_id}}=\frac{m_0\prod_{i=1}^n m_i^{e_i}}{m_0^{\sum_{i=1}^n e_i}}.
\]
It follows that $k[S_d,x_0^{-d}]_0\cong k[\frac{m_1}{m_0},\ldots, \frac{m_n}{m_0}]$. Since $\frac{m_1}{m_0},\ldots, \frac{m_n}{m_0}$ each contain a distinct variable $x_i$, they are algebraically independent. Thus the coordinate ring of $U'_0$ is regular and we deduce that $V^d$ is smooth at $q$ whenever $q\in U'_0$. In particular, the points $\psi(p)$ with $p\in U_0$ are smooth points of $V^d$.

The  inclusion $\iota: k[S_d]\hookrightarrow S$  yields upon inverting powers of $x_0$ the identity map 
\[
 k\left [\frac{m_1}{m_0},\ldots, \frac{m_n}{m_0}\right]=k[S_d, x_0^{-d}]\to S[x_0^{-1}]=k\left[\frac{x_1}{x_0},\ldots, \frac{x_n}{x_0}\right]
 \]
  since $m_i/m_0\mapsto x_i/x_0=m_i/m_0$ for $1\leq i\leq n$.

(2) The point $\psi(p)$ is defined since $m_0(p)\neq 0$.

Consider the affine cone $\widetilde{V^d}$ of $V^d$ and the regular map
 \[
 g: \AA_k^{n+1} \to \AA_k^{N+1},  g(p)=(m_0(p), \ldots, m_N(p)).
 \]
 The map induced by $g$ on coordinate rings is denoted $\theta$ in \Cref{eq: veronese map}.
 
Let $\fm$ denote the ideal of $p$ in  $\AA_k^{n+1}$, $\fm'$ denote the ideal of $g(p)$ in  $\AA_k^{N+1}$, and $\fm''$ denote the ideal of $g(p)$ in  $A'=k[\widetilde{V^d}]$. Let $\pi:k[\AA_k^{N+1}]\to A'$ denote the canonical projection and $\iota:k[S_d]\to S=k[A^{n+1}]$ denote the inclusion. These maps  satisfy the identity $\theta= \iota\circ\pi$. There are dual commutative diagrams as follows
 \[
 \begin{tikzcd}[column sep=1.5em]
 &\fm'/\fm'^2  \arrow{dl}{\theta}   \arrow{dr}{\pi} \\
 \fm/\fm^2 && \fm''/\fm''^2  \arrow{ll}{\iota}
\end{tikzcd}
\quad
 \begin{tikzcd}[column sep=1.5em]
 &T_{g(p)}\AA^{N+1}   \\
 T_p \AA^{n+1}  \arrow{ur}{\d g}  \arrow{rr}{\d \iota}  && T_{g(p)} \widetilde{V^d} \arrow{ul}{\d \pi}.
\end{tikzcd}
\]
 By part (1), the horizontal map labeled $\iota$ and hence also its dual $\d \iota$ are isomorphisms. Since $\pi$ is surjective, $\d \pi$ is injective. Hence $T_{g(p)} \widetilde{V^d}$ is isomorphic to the image of $\d \pi$ which coincides with the image of $\d g$.
 By \Cref{lem: differential map}, the differential $\d g: T_p \AA_k^{n+1} \to  T_{g(p)} \AA_k^{N+1}$ is represented by the Jacobian matrix 
 $J=\begin{bmatrix} \frac{\partial m_i}{\partial x_j}\end{bmatrix}_{\substack{0\leq i\leq N\\0\leq j\leq n}}$.
 Therefore,  $T_{\tilde{\psi(p)}}(\widetilde{V^d})$ can be identified with the column space of $J$ and its projectivization yields 
\[
\T_{\psi(p)}(V^d) \cong \PP \left \langle \left[ \frac{\partial m_0}{\partial{x_j}} \big|_p, \ldots, \frac{\partial m_N}{\partial{x_j}} \big |_p \right] \mid 0\leq j\leq n \right\rangle.
\]
\end{proof}

In the following, we to denote by $\widetilde{V^d}$ the affine cone over $V^d$ and by $\widetilde{q}$ any affine point that corresponds to the projective point $q\in \PP_k^N$. \Cref{prop: Lasker} in standard graded case was first proven by Lasker in \cite{Lasker}. Our proof is an adaption from \cite{BO08}.
\bpr [Lasker's proposition]\label{prop: Lasker}
Let $\PP=\PP(1,a_1,\ldots,a_n)$, let $d\geq a_n$ be an integer  and set $N=s_d-1$. 
Let $p\in U_0=\PP\setminus V(x_0)$. Then the  orthogonal  complement to the tangent space to $\widetilde{V^d}$ at $\widetilde{\psi(p)}$  consists of all hypersurfaces of degree $d$ singular at $p$. In detail, there is a vector space isomorphism
 \[
 [I_p^{(2)}]_d \cong T_{\widetilde{\psi(p)}}(\widetilde{V^d})^\perp.
 \] 
\epr

\bpf 
By the Zariski-Nagata Theorem, $[I_p^{(2)}]_d$ is the vector space of homogeneous polynomials of degree $d$ whose first derivative vanishes on $p$
\begin{equation*}
     (I_p^{(2)})_d  = \left \{  f \in S_d \; \big | \; \frac{\partial f}{\partial x_i}(p)=0 \right \} 
      =\left \{  \displaystyle \sum_{i=0}^N \lambda_i m_i \; \big | \; \sum_{i=0}^N \lambda_i\frac{\partial m_i}{\partial x_i}\big |_p=0 \right \}. 
\end{equation*}

Then for any $\underline{\lambda}\in k^{N+1}$ we have
\begin{equation*}
\begin{split}
&\sum_{i=0}^N \lambda_i\frac{\partial m_i}{\partial x_i}\big |_p=0\\
 \Leftrightarrow &
 (\lambda_0,\ldots , \lambda_N)\cdot\left (\frac{\partial m_1}{\partial x_i}\big |_p,\cdots , \frac{\partial m_N}{\partial x_i}\big |_p\right )=0  \\
\Leftrightarrow &\underline{\lambda} \in  T_{\widetilde{\psi_0(p)}}(\widetilde{V^d})^\perp.  \qedhere
\end{split}
\end{equation*}
\epf 

We now define the $k$-th secant variety of a projective variety. This is a central tool in establishing \Cref{equiv}. For a more thorough study of the secant variety see \cite[Lecture 8]{Harris}. 

\bdf[Secant Variety] \label{def: secant variety}
	Let $X$  be a projective variety. For any nonnegative integer $r$, the \emph{$r$-secant variety} of $X$, denoted by $\sigma_r(X)$, is defined to be
	$$\sigma_r(X) = \overline{\bigcup_{P_1, \dots, P_r \in X} \langle P_1, \dots, P_r\rangle}^{\text{ Zariski closure}}.$$
\edf 
\noindent Note that $\sigma_r(V^n_d)$ is an irreducible variety for all $r$. (\cite[Remark 1.2]{Fantechi})
\begin{lemma}[First Terracini Lemma] \label{1st Terracini Lemma}
Let $Y$ be an a affine variety set, let $p_1,\ldots,p_r \in Y$ be general points and $z\in \left < p_1,\ldots,p_r\right >$ be a general point. Then 
\[
T_{z}(\sigma_{r}(Y))=\left <T_{p_1}(Y),\ldots,T_{p_r}(Y) \right >.
\]
\end{lemma}
\bpf 
See \cite[Lemma 2.2]{BO08}.
\epf

Using arguments identical to  \cite{BO08} we prove \Cref{equiv}. 

\bpf [Proof of \Cref{equiv}]
Let $X=\{2p_1,\ldots , 2p_r \}$ be a collection of general double points in $\PP$. Applying \Cref{prop: Lasker} to get the second equality, we obtain
\begin{equation*}
   [ I_X]_d= \left[\bigcap_{i=1}^r I_{p_i}^{(2)}\right]_d = \displaystyle \bigcap_{i=1}^r \left (  T_{\widetilde{\psi(p_i)}}\widetilde{V^d} \right )^{\bot}= \left < T_{\widetilde{\psi(p_1)}}\widetilde{V^d},\ldots,T_{\widetilde{\psi(p_r)}}\widetilde{V^d} \right >  ^{\bot }.
\end{equation*}
Therefore, using \Cref{1st Terracini Lemma}, we obtain
\begin{equation*}
\dim([S/ I_X]_d)=\dim_k \left \langle T_{\widetilde{\psi(p_1)}}\widetilde{V^d},\ldots,T_{\widetilde{\psi(p_r)}}\widetilde{V^d} \right \rangle =\dim_k T_{z}(\sigma_{k}(\widetilde{V^d})). 
\end{equation*}
Finally using that $\sigma_r(\widetilde{V^d})=\widetilde{\sigma_r(V^d)}$ and 
the irreducibility of $\widetilde{V^d}$, which implies irreducibility of $\sigma_r(\widetilde{V^d})$ for all $r$ \cite[p. 144, Prop 11.24]{Harris}, we get for $z$  a general point of $\sigma_r(\widetilde{V^d})$
\begin{equation*}
  \dim T_{z}(\sigma_k(\widetilde{V^d})) =\dim \sigma_k(\widetilde{V^d}) =\dim(\sigma_k(V^d))+1.
\end{equation*}
To summarize, we have obtained the identity
\[
\dim([S/ I_X]_d)=\dim(\sigma_k(V^d))+1,
\]
whence the claimed equivalence follows.
\epf

\bibliographystyle{amsalpha}
\bibliography{biblio}

\newcommand{\etalchar}[1]{$^{#1}$}
\providecommand{\bysame}{\leavevmode\hbox to3em{\hrulefill}\thinspace}
\providecommand{\MR}{\relax\ifhmode\unskip\space\fi MR }
\providecommand{\MRhref}[2]{%
  \href{http://www.ams.org/mathscinet-getitem?mr=#1}{#2}
}
\providecommand{\href}[2]{#2}
\begin{thebibliography}{CHHVT20}

\bibitem[AH92a]{AH92b}
J.~Alexander and A.~Hirschowitz, \emph{La m\'{e}thode d'{H}orace \'{e}clat\'{e}e: application \`a l'interpolation en degr\'{e} quatre}, Invent. Math. \textbf{107} (1992), no.~3, 585--602. \MR{1150603}

\bibitem[AH92b]{AH92a}
\bysame, \emph{Un lemme d'{H}orace diff\'{e}rentiel: application aux singularit\'{e}s hyperquartiques de {${\bf P}^5$}}, J. Algebraic Geom. \textbf{1} (1992), no.~3, 411--426. \MR{1158623}

\bibitem[AH95]{AH95}
\bysame, \emph{Polynomial interpolation in several variables}, J. Algebraic Geom. \textbf{4} (1995), no.~2, 201--222. \MR{1311347}

\bibitem[Ale88]{Alexander88}
J.~Alexander, \emph{Singularit\'{e}s imposables en position g\'{e}n\'{e}rale \`a une hypersurface projective}, Compositio Math. \textbf{68} (1988), no.~3, 305--354. \MR{971330}

\bibitem[Amr94]{WPSKTheory}
Abdallah~Al Amrani, \emph{Complex {K}-theory of weighted projective spaces}, Journal of Pure and Applied Algebra \textbf{93} (1994), no.~2, 113--127.

\bibitem[BFNR13]{WPSAlgTop}
Anthony Bahri, Matthias Franz, Dietrich Notbohm, and Nigel Ray, \emph{The classification of weighted projective spaces}, Fundamenta Mathematicae \textbf{220} (2013), no.~3, 217–226.

\bibitem[BO08]{BO08}
Maria~Chiara Brambilla and Giorgio Ottaviani, \emph{On the {A}lexander-{H}irschowitz theorem}, J. Pure Appl. Algebra \textbf{212} (2008), no.~5, 1229--1251. \MR{2387598}

\bibitem[BR86]{WPSBeltRobb}
Mauro Beltrametti and Lorenzo Robbiano, \emph{Introduction to the theory of weighted projective spaces}, Expositiones Mathematicae \textbf{4} (1986).

\bibitem[BR07]{CCD}
Matthias Beck and Sinai Robins, \emph{Computing the continuous discretely}, Undergraduate Texts in Mathematics, Springer, New York, 2007, Integer-point enumeration in polyhedra. \MR{2271992}

\bibitem[Bre79]{Bre1}
H.~Bresinsky, \emph{Monomial space curves in {${\bf A}\sp{3}$} as set-theoretic complete intersections}, Proc. Amer. Math. Soc. \textbf{75} (1979), no.~1, 23--24. \MR{529205}

\bibitem[Cer17]{AlgebraicStatistics}
Carlos Enrique~Am{\'e}ndola Cer{\'o}n, \emph{Algebraic statistics of gaussian mixtures}, Ph.D. thesis, Technische Universit{\"a}t Berlin, 2017.

\bibitem[Cha01]{Chandler01}
Karen~A. Chandler, \emph{A brief proof of a maximal rank theorem for generic double points in projective space}, Trans. Amer. Math. Soc. \textbf{353} (2001), no.~5, 1907--1920. \MR{1813598}

\bibitem[Cha02]{Chandler02}
\bysame, \emph{Linear systems of cubics singular at general points of projective space}, Compositio Math. \textbf{134} (2002), no.~3, 269--282. \MR{1943904}

\bibitem[CHHVT20]{IdealsofPowers}
E.~Carlini, H.T. H{\`a}, B.~Harbourne, and A.~Van~Tuyl, \emph{Ideals of powers and powers of ideals: Intersecting algebra, geometry, and combinatorics}, Lecture Notes of the Unione Matematica Italiana, Springer International Publishing, 2020.

\bibitem[Den03]{Denham}
Graham Denham, \emph{Short generating functions for some semigroup algebras}, Electron. J. Combin. \textbf{10} (2003), Research Paper 36, 7. \MR{2014523}

\bibitem[Dol82]{DolgachevWPS}
Igor Dolgachev, \emph{Weighted projective varieties}, Group Actions and Vector Fields (Berlin, Heidelberg) (James~B. Carrell, ed.), Springer Berlin Heidelberg, 1982, pp.~34--71.

\bibitem[DREE{\etalchar{+}}23]{NumericalAG}
Sandra Di~Rocco, Parker~B Edwards, David Eklund, Oliver G{\"a}fvert, and Jonathan~D Hauenstein, \emph{Computing geometric feature sizes for algebraic manifolds}, SIAM Journal on Applied Algebra and Geometry \textbf{7} (2023), no.~4, 716--741.

\bibitem[Fan90]{Fantechi}
Barbara Fantechi, \emph{On the superadditivity of secant defects}, Bulletin de la Soci\'et\'e Math\'ematique de France \textbf{118} (1990), no.~1, 85--100 (en). \MR{92c:14049}

\bibitem[GKW91]{KunzMonomial}
W.~Gastinger, E.~Kunz, and R.~Waldi, \emph{Relation matrices of monomial curves}, An. Univ. Bucure\c{s}ti Mat. \textbf{40} (1991), no.~1-2, 41--53. \MR{1220263}

\bibitem[Gro60]{EGA2}
A.~Grothendieck, \emph{\'{E}l\'{e}ments de g\'{e}om\'{e}trie alg\'{e}brique. {I}. {L}e langage des sch\'{e}mas.}, Inst. Hautes \'{E}tudes Sci. Publ. Math. (1960), no.~4, 228. \MR{217083}

\bibitem[GS]{M2}
Daniel~R. Grayson and Michael~E. Stillman, \emph{Macaulay2, a software system for research in algebraic geometry}, Available at \url{http://www.math.uiuc.edu/Macaulay2/}.

\bibitem[Har92]{Harris}
Joe Harris, \emph{Algebraic geometry}, Graduate Texts in Mathematics, vol. 133, Springer-Verlag, New York, 1992, A first course. \MR{1182558}

\bibitem[Her70]{Herzog}
J\"{u}rgen Herzog, \emph{Generators and relations of abelian semigroups and semigroup rings}, Manuscripta Math. \textbf{3} (1970), 175--193. \MR{269762}

\bibitem[Hir85]{Hirschowitz85}
Andr\'{e} Hirschowitz, \emph{La m\'{e}thode d'{H}orace pour l'interpolation \`a plusieurs variables}, Manuscripta Math. \textbf{50} (1985), 337--388. \MR{784148}

\bibitem[HM21]{HM21}
Huy~T\`ai H\`a and Paolo Mantero, \emph{The {A}lexander-{H}irschowitz theorem and related problems}, Commutative algebra, Springer, Cham, [2021] \copyright 2021, pp.~373--427. \MR{4394415}

\bibitem[Hui18]{JackHuizenga}
Jack Huizenga, \emph{\href{chrome-extension://efaidnbmnnnibpcajpcglclefindmkaj/https://sites.psu.edu/jhuizenga/files/2023/06/InterpolationNotes.pdf}{Polynomial Interpolation: An introduction to algebraic geometry}}, 2018.

\bibitem[KTB19]{NeuralNetwork}
Joe Kileel, Matthew Trager, and Joan Bruna, \emph{On the expressive power of deep polynomial neural networks}, Advances in neural information processing systems \textbf{32} (2019).

\bibitem[Las04]{Lasker}
Emanuel Lasker, \emph{Zur {T}heorie der kanonischen {F}ormen}, Math. Ann. \textbf{58} (1904), no.~3, 434--440. \MR{1511244}

\bibitem[Pal03]{Platini1}
F.~Palatini, \emph{Sulla rappresentazione delle forme ternarie me\-dian\-te la somma di potenze di forme lineari}, Rend. Accad. Lincei V. (1903), no.~12, 378--384 (Italian).

\bibitem[Pos12]{AHElisaPostinghel}
Elisa Postinghel, \emph{A new proof of the {A}lexander-{H}irschowitz interpolation theorem}, Annali di Matematica Pura ed Applicata \textbf{191} (2012), no.~1, 77--94.

\bibitem[RT11]{WPSToric}
Michele Rossi and Lea Terracini, \emph{Weighted projective spaces from the toric point of view with computational applications}, \href{https://arxiv.org/abs/1112.1677}{arXiv:1112.1677}, 2011.

\bibitem[Sen20]{affinemonomialcurves}
Indranath Sengupta, \emph{Affine monomial curves}, 2020.

\bibitem[Sta99]{stanley1}
Richard~P. Stanley, \emph{Enumerative combinatorics}, vol.~2, Cambridge University Press, 1999.

\bibitem[Ter11]{Terra1}
Alessandro Terracini, \emph{Sulle $v_k$ per cui la variet\`a degli ${S}_h$ $h+1$-secanti ha dimensione minore dell'ordinario}, Rend. Circ. Mat. Palermo Selecta Vol. \textbf{I} (1911), no.~31 (Italian).

\bibitem[Ter15a]{Terra3}
A.~Terracini, \emph{Sulla rappresentazione delle forme quaternarie mediante somme di potenze di forme lineari}, Atti R. Accad. delle Scienze di Torino. Selecta Vol. \textbf{Selecta I} (1915), no.~51 (Italian).

\bibitem[Ter15b]{Terra2}
Alessandro Terracini, \emph{Sulla rappresentazione delle coppie di forme ternarie mediante somme di potenze di forme lineari}, Annali Mat. Selecta Vol. \textbf{I} (1915), no.~24 (Italian).

\end{thebibliography}

\end{document}